\documentclass [12pt]{amsart}
\usepackage{amssymb,amsmath,amsthm}
\newtheorem{theorem}{Theorem}
\newtheorem{lemma}{Lemma}
\newtheorem{proposition}{Proposition}

\newcommand{\ad}{\,\mathrm{ad}\,}

\newcommand{\GL}{\,\mathrm{GL}\,}
\newcommand{\SL}{\,\mathrm{SL}\,}

\newcommand{\diag}{\,\mathrm{diag}\,}
\begin{document}

\begin{center}

{\Large {\bf Automorphisms of Chevalley groups of type $B_l$\\

\medskip

 over local rings with $1/2$\footnote{The
work is supported by the Russian President grant MK-2530.2008.1 and
by the grant of Russian Fond of Basic Research 08-01-00693. }} }

\bigskip

{\large \bf E.~I.~Bunina}

\end{center}
\bigskip

\begin{center}

{\bf Abstract.}

\end{center}

In the given paper we prove that every automorphism of a Chevalley group of type $B_l$, $l\geqslant 2$, over a commutative local ring with~$1/2$ is standard, i.\,e., it is a composition of ring, inner and central automorphisms.

\smallskip

\section*{Introduction}\leavevmode

An associative commutative ring $R$ with a unit is called  \emph{local},
if it contains exactly one maximal ideal (that coincides with the radical of~$R$). Equivalently, the set of all non-invertible
elements of~$R$ is an ideal.

We describe automorphisms of Chevalley groups of type $B_l$ over local rings with~$1/2$.

The analogue of Theorem~1 for the root systems $A_l, D_l, $ and $E_l$ was obtained by the author in~\cite{ravnyekorni}, in~\cite{normalizers} all automorphisms of Chevalley groups of given types over local rings with~$1/2$ were described.  Theorem~1 for the root systems $B_2$ and $G_2$ is proved in~\cite{korni2},  for the system $F_4$ it is proved in~\cite{bunF4}.
.

Similar results for Chevalley groups over fields were proved
 by R.\,Steinberg~\cite{Stb1} for the finite case and by J.\,Humphreys~\cite{H} for the infinite case. Many papers were devoted
to description of automorphisms of Chevalley groups over different
commutative rings, we can mention here the papers of
Borel--Tits~\cite{v22}, Carter--Chen~Yu~\cite{v24},
Chen~Yu~\cite{v25}--\cite{v29}, A.\,Klyachko~\cite{Klyachko}.
 E.\,Abe~\cite{Abe_OSN} proved that all automorphisms of Chevalley groups under Noetherian  rings with~$1/2$ are standard.

The case
$A_l$ was completely studied by the papers of
W.C.~Waterhouse~\cite{v46}, V.M.~Petechuk~\cite{v12},  Fuan Li and
Zunxian Li~\cite{v37}, and also for rings without~$1/2$. The paper
of I.Z.\,Golubchik and A.V.~Mikhalev~\cite{v8} covers the
case~$C_l$, that is not considered in the present paper. Automorphisms and isomorphisms of general linear groups over arbitrary associative rings were described by E.I.~Zelmanov in~\cite{v11} and by I.Z.~Golubchik, A.V.~Mikhalev in~\cite{GolMikh1}.

 We generalize some methods of V.M.~Petechuk~\cite{Petechuk1} to prove Theorem~1.

The author is thankful to N.A.\,Vavilov,  A.A.\,Klyachko,
A.V.\,Mikhalev for valuable advices, remarks and discussions.

\section{Definitions and main theorems.}\leavevmode

 We fix the root system~$\Phi$ of the type $B_l$, $l\geqslant 2$. Detailed texts about root
systems and their properties can be found in the books
\cite{Hamfris}, \cite{Burbaki}).

Let $e_1,\dots,e_l$
be an orthonorm basis of the space $\mathbb R^l$. Then we  numerate the roots of $B_l$
as follows:
$$
\alpha_1= e_1-e_2, \alpha_2=e_2-e_3, \alpha_{l-1}=e_{l-1}-e_l, \alpha_l=e_l
$$
are simple roots;
\begin{align*}
e_i\pm e_j, \ e_i, \ i<j,
\end{align*}
are other positive roots.

Suppose now that we have a
semisimple complex Lie algebra~$\mathcal L$ of type $\Phi$ with
Cartan subalgebra~$\mathcal H$ (detailed information about
semisimple Lie algebras can be found in the book~\cite{Hamfris}).

Then we can choose a basis $\{ h_1, \dots, h_l\}$ in~$\mathcal H$ and for every
$\alpha\in \Phi$ elements $x_\alpha \in {\mathcal L}_\alpha$ so that $\{ h_i; x_\alpha\}$ form a basis in~$\mathcal L$ and for every two elements of this basis their commutator is an integral linear combination of the elements of the same basis.

Let us introduce elementary Chevalley groups (see,
for example,~\cite{Steinberg}).

Let  $\mathcal L$ be a semisimple Lie algebra (over~$\mathbb C$)
with a root system~$\Phi$, $\pi: {\mathcal L}\to gl(V)$ be its
finitely dimensional faithful representation  (of dimension~$n$). If
$\mathcal H$ is a Cartan subalgebra of~$\mathcal L$, then a
functional
 $\lambda \in {\mathcal H}^*$ is called a
 \emph{weight} of  a given representation, if there exists a nonzero vector $v\in V$
 (that is called a  \emph{weight vector}) such that
for any $h\in {\mathcal H}$ $\pi(h) v=\lambda (h)v.$

In the space~$V$ there exists a basis of weight vectors such that
all operators $\pi(x_\alpha)^k/k!$ for $k\in \mathbb N$ are written
as integral (nilpotent) matrices. This basis is called a
\emph{Chevalley basis}. An integral matrix also can be considered as
a matrix over an arbitrary commutative ring with~$1$. Let $R$ be
such a ring. Consider matrices $n\times n$ over~$R$, matrices
$\pi(x_\alpha)^k/k!$ for
 $\alpha\in \Phi$, $k\in \mathbb N$ are included in $M_n(R)$.

Now consider automorphisms of the free module $R^n$ of the form
$$
\exp (tx_\alpha)=x_\alpha(t)=1+tx_\alpha+t^2 (x_\alpha)^2/2+\dots+
t^k (x_\alpha)^k/k!+\dots
$$
Since all matrices $x_\alpha$ are nilpotent, we have that this
series is finite. Automorphisms $x_\alpha(t)$ are called
\emph{elementary root elements}. The subgroup in $Aut(R^n)$,
generated by all $x_\alpha(t)$, $\alpha\in \Phi$, $t\in R$, is
called an \emph{elementary adjoint Chevalley group} (notation:
$E_{\ad}(\Phi,R)=E_{\ad}(R)$).

The action of $x_\alpha(t)$ on the Chevalley basis is described in
\cite{v23}, \cite{VavPlotk1}.

All weights of a given representation (by addition) generate a
lattice (free Abelian group, where every  $\mathbb Z$-basis  is also
a $\mathbb C$-basis in~${\mathcal H}^*$), that is called the
\emph{weight lattice} $\Lambda_\pi$.

Elementary Chevalley groups are defined not even by a representation
of the Chevalley groups, but just by its \emph{weight lattice}.
Namely, up to an abstract isomorphism an elementary Chevalley group
is completely defined by a root system~$\Phi$, a commutative
ring~$R$ with~$1$ and a weight lattice~$\Lambda_\pi$.

Among all lattices we can mark  the lattice corresponding to the
adjoint representation: it is generated by all roots (the \emph{root
lattice}~$\Lambda_{ad}$). The corresponding (elementary) Chevalley group is called \emph{adjoint}.

Introduce now Chevalley groups (see~\cite{Steinberg},
\cite{Chevalley}, \cite{v3}, \cite{v23}, \cite{v30}, \cite{v43},
\cite{VavPlotk1}, and also latter references in these papers).

Consider semisimple linear algebraic groups over algebraically
closed fields. These are precisely elementary Chevalley groups
$E_\pi(\Phi,K)$ (see.~\cite{Steinberg},~\S\,5).

All these groups are defined in $SL_n(K)$ as  common set of zeros of
polynomials of matrix entries $a_{ij}$ with integer coefficients
 (for example,
in the case of the root system $C_l$ and the universal
representation we have $n=2l$ and the polynomials from the condition
$(a_{ij})Q(a_{ji})-Q=0$). It is clear now that multiplication and
taking inverse element are also defined by polynomials with integer
coefficients. Therefore, these polynomials can be considered as
polynomials over arbitrary commutative ring with a unit. Let some
elementary Chevalley group $E$ over~$\mathbb C$ be defined in
$SL_n(\mathbb C)$ by polynomials $p_1(a_{ij}),\dots, p_m(a_{ij})$.
For a commutative ring~$R$ with a unit let us consider the group
$$
G(R)=\{ (a_{ij})\in \SL_n(R)\mid \widetilde p_1(a_{ij})=0,\dots
,\widetilde p_m(a_{ij})=0\},
$$
where  $\widetilde p_1(\dots),\dots \widetilde p_m(\dots)$ are
polynomials having the same coefficients as
$p_1(\dots),\dots,p_m(\dots)$, but considered over~$R$.

This group is called the \emph{Chevalley group} $G_\pi(\Phi,R)$ of
the type~$\Phi$ over the ring~$R$, and for every algebraically
closed field~$K$ it coincides with the elementary Chevalley group.

The subgroup of diagonal (in the standard basis of weight vectors)
matrices of the Chevalley group $G_\pi(\Phi,R)$ is called the
 \emph{standard maximal torus}
of $G_\pi(\Phi,R)$ and it is denoted by $T_\pi(\Phi,R)$. This group
is isomorphic to $Hom(\Lambda_\pi, R^*)$.

Let us denote by $h(\chi)$ the elements of the torus $T_\pi
(\Phi,R)$, corresponding to the homomorphism $\chi\in Hom
(\Lambda(\pi),R^*)$.

In particular, $h_\alpha(u)=h(\chi_{\alpha,u})$ ($u\in R^*$, $\alpha
\in \Phi$), where
$$
\chi_{\alpha,u}: \lambda\mapsto u^{\langle
\lambda,\alpha\rangle}\quad (\lambda\in \Lambda_\pi).
$$

Note that the condition
$$
G_\pi (\Phi,R)=E_\pi (\Phi,R)
$$
is not true even for fields, that are not algebraically closed.
Let us show the difference between Chevalley groups and their elementary subgroups in the case when $R$ is semilocal. In this case
 $G_\pi (\Phi,R)=E_\pi(\Phi,R)T_\pi(\Phi,R)$
(see~\cite{v38}, \cite{Abe1}, \cite{v19}), and elements $h(\chi)$ are connected with elementary generators
by the formula
\begin{equation}\label{ee4}
h(\chi)x_\beta (\xi)h(\chi)^{-1}=x_\beta (\chi(\beta)\xi).
\end{equation}

 In the case of semilocal rings from  the formula~\eqref{ee4} we see that
$$
[G(\Phi,R),G(\Phi,R)]\subseteq E(\Phi,R).
$$
If $R$ (as in our case) also contains~$1/2$,
$l\geqslant 2$, we can easily show that
$$
[G(\Phi,R),G(\Phi,R)]=[E(\Phi,R), E(\Phi,R)]=E(\Phi,R).
$$

Define four types of automorphisms of a Chevalley group
 $G_\pi(\Phi,R)$, we
call them  \emph{standard}.

{\bf Central automorphisms.} Let $C_G(R)$ be a center of
$G_\pi(\Phi,R)$, $\tau: G_\pi(\Phi,R) \to C_G(R)$ be some
homomorphism of groups. Then the mapping $x\mapsto \tau(x)x$ from
$G_\pi(\Phi,R)$ onto itself is an automorphism of $G_\pi(\Phi,R)$,
that is denoted by~$\tau$ and called a \emph{central automorphism}
of the group~$G_\pi(\Phi,R)$.

Note that every central automorphism of a Chevalley group
$G_\pi(\Phi,R)$ is identical on the commutant. Under our conditions
the elementary subgroup $E_\pi (\Phi,R)$ is a commutant of the groups $G_\pi(\Phi,R)$ and $E_\pi(\Phi,R)$,
therefore on elementary Chevalley groups all central automorphisms are identical.

{\bf Ring automorphisms.} Let $\rho: R\to R$ be an automorphism of
the ring~$R$. The mapping $(a_{i,j})\mapsto (\rho (a_{i,j}))$ from $G_\pi(\Phi,R)$
onto itself is an automorphism of the group $G_\pi(\Phi,R)$, that is
denoted by the same letter~$\rho$ and is called a \emph{ring
automorphism} of the group~$G_\pi(\Phi,R)$. Note that for all
$\alpha\in \Phi$ and $t\in R$ an element $x_\alpha(t)$ is mapped to
$x_\alpha(\rho(t))$.

{\bf Inner automorphisms.} Let $S$ be some ring containing~$R$,  $g$
be an element of $G_\pi(\Phi,S)$, that normalizes the subgroup $G_\pi(\Phi,R)$. Then
the mapping $x\mapsto gxg^{-1}$  is an automorphism
of the group~$G_\pi(\Phi,R)$, that is denoted by $i_g$ and is called an
\emph{inner automorphism}, \emph{induced by the element}~$g\in G_\pi(\Phi,S)$. If $g\in G_\pi(\Phi,R)$, then call $i_g$ a \emph{strictly inner}
automorphism.

{\bf Graph automorphisms.} Let $\delta$ be an automorphism of the
root system~$\Phi$ such that $\delta \Delta=\Delta$. Then there
exists a unique automorphisms of $G_\pi (\Phi,R)$ (we denote it by
the same letter~$\delta$) such that for every $\alpha \in \Phi$ and
$t\in R$ an element $x_\alpha (t)$ is mapped to
$x_{\delta(\alpha)}(\varepsilon(\alpha)t)$, where
$\varepsilon(\alpha)=\pm 1$ for all $\alpha \in \Phi$ and
$\varepsilon(\alpha)=1$ for all $\alpha\in \Delta$.

There no graph automorphisms in the case $B_l$ under consideration.

Similarly we can define four type of automorphisms of the elementary
subgroup~$E(R)$. An automorphism~$\sigma$ of the group
 $G_\pi(\Phi,R)$ (or $E_\pi(\Phi,R)$)
is called  \emph{standard} if it is a composition of automorphisms
of these introduced four types.

Together with standard automorphisms we use the following ''temporary'' type of automorphisms of an elementary adjoint Chevalley group:

{\bf Automorphisms--conjugations.} Let $V$ be the representation space of the group $E_{\ad} (\Phi,R)$, $C\in \GL(V)$ be some matrix from its normalizer:
$$
C E_{\ad}(\Phi,R) C^{-1}= E_{\ad} (\Phi,R).
$$
 Then the mapping
 $x\mapsto CxC^{-1}$ from $E_\pi(\Phi,R)$ onto itself is an automorphism of the Chevalley group, that is denoted by~$i_С$ and is called an  \emph{automorphism--conjugation} of~$E(R)$,
\emph{induced by an element}~$C$ of~$\GL(V)$.

Our aim is to prove the next main theorem:

\begin{theorem}\label{main}
Let $G=G_{\pi}(\Phi,R)$ $(E_\pi(\Phi,R))$
be an \emph{(}elementary\emph{)} Chevalley group with a the root system  $B_l$, $l\geqslant 2$, $R$ is a commutative local ring
with~$1/2$. Then every automorphism of~$G$ is standard. If the Chevalley group is adjoint, then an inner automorphism in the composition is strictly
inner.
\end{theorem}

The main theorem follows from the next two theorems:

\begin{theorem}\label{old}
Every automorphism of an elementary adjoint Chevalley group of type
$B_l$, $l\geqslant 2$, over a local ring with~$1/2$
is a composition of a ring automorphism and an automorphism--conjugation.
\end{theorem}

\begin{theorem}\label{norm}
Every  automorphism--conjugation of an elementary adjoint Chevalley group  of type $B_l$, $l\geqslant 2$, over a local ring is  strictly inner \emph{(}conjugation with the help of the corresponding Chevalley group\emph{)}.
\end{theorem}

Four next sections a devoted to the proof of Theorem~\ref{old}, \S\,6 is devoted to the proof of Theorem~\ref{norm}.

\section{Replacing the initial automorphism to the special one.}\leavevmode

In this section we use some reasonings
from~\cite{Petechuk1}.

Let $J$ be the maximal ideal (radical) of~$R$, $k$  the residue
field $R/J$. Then $E_J=E_{ad}( \Phi,R,J)$ is the greatest normal
proper subgroup of $E_{\ad}(\Phi,R)$ (see~\cite{Abe1}).

Therefore,
$E_J$ is invariant under the action of~$\varphi$.

By this reason  the automorphism
$$
\varphi: E_{\ad} (\Phi,R)\to E_{\ad}(\Phi,R)
$$
induces an automorphism
$$
\overline \varphi: E_{\ad} (\Phi,R)/E_J=E_{\ad} (\Phi,k)\to
E_{\ad}(\Phi,k).
$$
The group $E_{\ad}(\Phi,k)$ is a Chevalley group over field,
therefore the automorphism $\overline \varphi$ is standard, i.\,e. it has the form
$$
\overline \varphi =   i_{\overline g} \overline \rho,\quad \overline
g\in E_J(E_{\ad}(\Phi,k)).
$$
(for the type $B_l$ graph automorphisms do not exist).

 It is clear that there exists a matrix $g\in GL_n(R)$ such that
its image under factorization  $R$ by~$J$ coincides with~$\overline
g$. We are not sure that  $g\in N_{\GL_n(R)}(E_{\ad}(\Phi,R))$.

Consider  a mapping $\varphi'= i_{g^{-1}} \varphi$. It is an
isomorphism of the group
 $E_{ad}(\Phi,R)\subset GL_n(R)$ onto some subgroup in $GL_n(R)$,
with the property that its image under factorization $R$ by $J$
coincides with the automorphism $\overline \rho$.

It gives us that Every matrix $A\in E_{\ad}(\Phi,R)$ with elements from the
subring~$R'$ of~$R$, generated by unit, is mapped under the action
of~$\varphi'$ to some matrix from the set~$A\cdot \GL_n(R,J)$.

No use some arguments from the paper~\cite{Petechuk1} (the group $\SL_n(R)$ was considered there).

Let $a\in E_{ad} (\Phi,R)$, $a^2=1$. Then the element $e=\frac{1}{2}
(1+a)$ is an idempotent in the ring $M_n(R)$. This idempotent  $e$
defines a decomposition of the free $R$-module $V=R^n$:
$$
V=eV\oplus (1-e)V=V_0\oplus V_1.
$$
With the isomorphism $\varphi'$ we associate the decomposition $V=V_0'\oplus V_1'$
of our module~$V$, defined by an idempotent  $e'=\frac{1}{2}
(1+\varphi(a)).$ Let $\overline V=\overline V_0 \oplus \overline
V_1= \overline V_0'\oplus \overline V_1'$
be decompositions of a $k$-module~$\overline V$ with respect to~$\overline a$ and
$\varphi(\overline a)$, and $\overline e=\frac{1}{2} (1+\overline a)$,
 $\overline e'=\frac{1}{2} (1+\overline \varphi(\overline a))$.

Then we have
\begin{proposition}\emph{(see Proposition~6 in~\cite{Petechuk1})}.\label{pr1_1}
The modules \emph{(}subspaces\emph{)}
 $\overline V_0$, $\overline V_1$ are images of the modules $V_0$, $V_1$ under factorization by~$J$.
\end{proposition}
\begin{proof} Let us denote the images of $V_0$, $V_1$ under factorization
by $J$ by $\widetilde V_0$, $\widetilde V_1$, respectively. Since
$V_0=\{ x\in V| ex=x\},$ $V_1= \{ x\in V|ex=0\},$  we have
 $\overline e(\overline x)=\frac{1}{2}(1+\overline a)(\overline x)=\frac{1}{2}
(1+\overline a(\overline
x))=\frac{1}{2}(1+\overline{a(x)})=\overline{e(x)}$. Then
$\widetilde V_0\subseteq \overline V_0$, $\widetilde V_1\subseteq
\overline V_1$.

Let $x=x_0+x_1$, $x_0\in V_0$, $x_1\in V_1$. Then $\overline
e(\overline x)=\overline e(\overline x_0)+\overline e (\overline
x_1)=\overline x_0$. If $\overline x\in \widetilde V_0$, then
$\overline x=\overline x_0$.
\end{proof}

Similarly to Proposition~8 from~\cite{Petechuk1} we have

\begin{proposition}\label{pr1_2}
Suppose that an isomorphism  $\varphi': E_{\ad} (\Phi,R)\to
\varphi(E_{\ad}(\Phi,R))\subset \GL_n(R)$ satisfies all properties described above, $ a\in E_{\ad}(\Phi,R)$, $a^2=1$, $a$ is a matrix
with elements from the subring of~$R$, generated by the unit, $b$
and $a$ are equivalent modulo~$J$, $V=V_0\oplus V_1$ is a
decomposition of~$V$ with respect to~$a$, $V=V_0'\oplus V_1'$ is a
decomposition of~$V$ with respect to~$b$.

If $V_0,V_1$ are free modules, then $V_0', V_1'$ are also free modules and $\dim V_0'=\dim V_0$, $\dim V_1'=\dim V_1$.
\end{proposition}

\section{The images of~$w_{\alpha_i}$}

We consider some fixed adjoint Chevalley group $E=E_{ad}(\Phi,R)$
with the root system $B_l$ ($l\geqslant 2$),
 its adjoint representation in the group $GL_n(R)$
($n=l+2m$, where $m$ is the number of positive roots of~$\Phi$),
with the basis of weight vectors $v_1=x_{\alpha_1},
v_{-1}=x_{-\alpha_1}, \dots, v_n=x_{\alpha_n}, v_{-n}=x_{-\alpha_n},
V_1=h_{1},\dots,V_l=h_{l}$, corresponding to the Chevalley basis of
the system~$\Phi$.

We also have the isomorphism~$\varphi'$, described in Section~2.

Consider the matrices $h_{\alpha_1}(-1),\dots, h_{\alpha_l}(-1)$ in
our basis. They have the form
$$
h_{\alpha_i}(-1)=\diag [\pm 1,\dots, \pm 1,
\underbrace{1,\dots,1}_{l}],
$$
on $(2j-1)$-th and $(2j)$-th places we have $-1$ if and only if
 $\langle \alpha_i,\alpha_j\rangle=-1$. As we see, for all~$i$
$h_{\alpha_i}(-1)^2=1$.

According to Proposition~\ref{pr1_2} we know that every matrix
$h_i=\varphi''(h_{\alpha_i}(-1))$ in some basis is diagonal with
$\pm 1$ on the diagonal, and the number of $1$ and $-1$ coincides
with its number for the matrix $h_{\alpha_i}(-1)$. Since all
matrices $h_i$ commutes, there exists a basis, where all $h_i$ have
the same form as $h_{\alpha_i}(-1)$. Suppose that we come to this
basis with the help of the matrix~$g_1$. It is clear that
 $g_1\in GL_n(R,J)$. Consider the mapping
 $\varphi_1=i_{g_1}^{-1} \varphi'$. It is also an isomorphism
of the group $E$ onto some subgroup of $GL_n(R)$ such that its image
under factorization $R$ by~$J$ is~$\overline \rho$, and
$\varphi_1(h_{\alpha_i}(-1))=h_{\alpha_i}(-1)$ for all
$i=1,\dots,l$.

Let us consider the isomorphism~$\varphi_1$.

Every element $w_i=w_{\alpha_i}(1)$ moves by conjugation  $h_i$ to each other, therefore its image has a block-monomial form. In particular, this image can be rewritten as a block-diagonal matrix, where the first block is  $2m\times 2m$, the second is
$\l\times l$.

Consider the first basis vector after the last basis change. Denote it by~$e$. The Weil group  $W$
acts transitively on the set of roots of the same length, therefore for every
root~$\alpha_i$ of the same length as the first one,
there exists such $w^{(\alpha_i)}\in W$, that $w^{(\alpha_i)}
\alpha_1=\alpha_i$. Similarly, all roots of the second length are also conjugate under the action
of~$W$. Let $\alpha_k$ be the first root of the length that is not equal to the length of~$\alpha_1$, and let $f$ be
the $k$-th basis vector after the last basis change. If $\alpha_j$ is a root conjugate to $\alpha_k$, then let us denote by $w_{(\alpha_j)}$ an element of~$W$ such that
$w_{(\alpha_j)} \alpha_k=\alpha_j$. Consider now the basis
$e_1,\dots, e_{2m}, e_{2m+1},\dots, e_{2m+l}$, where $e_1=e$, $e_k=f$,
for $1< i\le 2m$ either $e_i=\varphi_1(w^{(\alpha_i)})e$, or
$e_i=\varphi_1(w_{(\alpha_i)})f$ (it depends of the length of $\alpha_k$); for $2m< i\leqslant 2m+l$ we do not move $e_i$. Clear that the matrix of this basis change is equivalent to the unit modulo radical. Therefore the obtained set of vectors also is a basis.

Clear that a matrix for $\varphi_1(w_i)$ ($i=1,\dots,l$) in the basis part $\{ e_1,\dots,e_{2m}\}$
 coincides with the matrix for $w_i$ in the initial basis of weight vectors.
Since $h_i(-1)$ are squares of $w_i$, then there images are not changed in the new basis.

Besides, we know that every matrix $\varphi_1(w_i)$
is block-diagonal up to decomposition of basis in the first  $2m$
and last $l$ elements. Therefore the last part of basis consisting of $l$ elements, can be changed independently.

Let us denote the elements and their images on this part of basis by $\widetilde w_i$ and $\widetilde{w_i'}$, respectively. These are matrices $l\times l$.

The matrix $\widetilde w_1$ is identical on the positions $3,\dots,l$, on the positions $1,2$ it has the form
$$
\begin{pmatrix}
-1& 1\\
0& 1
\end{pmatrix}.
$$
The matrices $\widetilde w_i$, $i=2,\dots,l-1$, are identical on the positions $1,\dots, i-2, i+2,\dots,l$, and on the positions $i-1,i,i+1$ they have the form
$$
\begin{pmatrix}
1& 0& 0\\
1& -1& 1\\
0& 0& 1
\end{pmatrix}.
$$
The matrix $\widetilde w_l$ is identical on the positions $1,\dots, l-2$, and on the positions $l-1,l$ it is
$$
\begin{pmatrix}
1& 0\\
2& -1
\end{pmatrix}.
$$

Let $\widetilde V=\widetilde
V_0^i\oplus \widetilde V_1^i$ be a decomposition of
$\widetilde{w_i'}$.

\begin{lemma}\label{l3_1}
Matrices $\widetilde{w_i'}$ and $\widetilde{w_j'}$,
$i\ne j$, commute if and only if $\widetilde
V_1^i\subseteq \widetilde V_0^j$ and $\widetilde V_1^j\subseteq
\widetilde V_0^i$.
\end{lemma}
\begin{proof}
If $\widetilde{\varphi_1(w_i)}$ and $\widetilde{\varphi_1(w_j)}$
commute, then the (free one-dimensional) submodule $\widetilde
V_1^i$ is proper for $\widetilde{\varphi_1(w_j)}$ and the (free
one-dimensional) submodule $\widetilde V_1^j$ is proper for
$\widetilde{\varphi_1(w_i)}$. Therefore either $\widetilde
V_1^i\subset \widetilde V_1^j$ or $\widetilde V_1^i\subset
\widetilde V_0^j$. If $\widetilde V_1^i\subset \widetilde V_1^j$
then $\widetilde V_1^i=\widetilde V_1^j$. Since the module $V_0^i$
is invariant for $\widetilde \varphi_1(w_j)$, we have $\widetilde
V_0^i\subset \widetilde V_0^j$, therefore $\widetilde
V_0^i=\widetilde V_0^j$, and so $\widetilde
\varphi_1(w_i)=\widetilde \varphi_1(w_j)$ and we come to
contradiction. Consequently, $\widetilde V_1^i\subset \widetilde
V_0^j$, and similarly $\widetilde V_1^j\subset \widetilde
V_0^i$.\end{proof}

\begin{lemma}\label{l3_2}
For the root system $B_l$ there exists such a basis in $\widetilde
V$ that the matrix $\widetilde \varphi_1(w_1)$ in this basis has the
same form as $w_1$, i.e. is equal to
$$
\begin{pmatrix}
-1& 1& 0\\
0& 1& 0\\
0& 0& E_{l-2}
\end{pmatrix}.
$$
\end{lemma}
\begin{proof}
Since $\widetilde w_1$ is an involution and $\widetilde V_1^1$ has the dimension $1$, then there exists a basis $\{ e_1,e_2,\dots, e_l\}$, such that $\widetilde {w_1'}$ in it has the form $\diag
[-1,1,\dots,1]$. In the basis $\{ e_1, e_2-1/2e_1,e_3,\dots, e_l\}$
the matrix $\widetilde {w_1'}$ has the obtained form.
\end{proof}

\begin{lemma}\label{l3_4}
For the root system $B_l$, $l>2$ if $\widetilde w_1=\widetilde{w_1'}$, \dots, $\widetilde w_{i-1}=\widetilde{w_{i-1}'}$, $i<l$, then we can choose a basis in
$\widetilde V$ such that in this basis  $\widetilde w_1=\widetilde{w_1'}$, \dots, $\widetilde w_{i-1}=\widetilde{w_{i-1}'}$,
$\widetilde{w_i'}=\widetilde w_i$.
\end{lemma}
\begin{proof}
An intersection of modules $\widetilde V_0^1$,  $\widetilde V_0^2$, \dots, $\widetilde V_0^i$ is a free module of dimension $\ge l-i-1$. Therefor we can suppose that $\widetilde
{w_2'}$ is $\begin{pmatrix} *& 0\\
0& E_{l-i-1} \end{pmatrix}$. Since $\widetilde{w_i'}$ commutes with $\widetilde w_1$,\dots, $\widetilde w_{i-1}$ and is an involution, then the matrix $\widetilde{w_i'}$ is identical on the first $i-2$ vectors. Now we can consider not all module~$\widetilde V$, but its restriction on three basic vectors with numbers $i-1,i,i+1$. Let
$$
\widetilde {w_i'}=\begin{pmatrix} a_1& a_2& a_3\\
b_1& b_2& b_3\\
c_1& c_2& c_3
\end{pmatrix}.
$$
Making the basis change with the matrix
$$
\begin{pmatrix}
E_{i-2}& 0& 0& 0& 0\\
0& b_1& (1-b_1)/2& 0& 0\\
0& 0& 1& 0& 0\\
0& 0& 0& 1& 0\\
0& 0& 0& 0& E_{l-i-1}
\end{pmatrix},
$$
we have that  $\widetilde w_1$,\dots, $\widetilde w_{i-1}$ are not moved, but $\widetilde
{w_2'}$ are moved to
$$
\widetilde \varphi_1(w_1)=\begin{pmatrix} a_1'& a_2'& a_3'\\
1& b_2'& b_3'\\
c_1'& c_2'& c_3'
\end{pmatrix}.
$$
Now use the following conditions: $\widetilde {w_i'}^2=E$ (cond. 1) and $\widetilde w_{i-1}\widetilde
{w_i'} \widetilde w_{i-1}=\widetilde {w_i'} \widetilde
w_{i-1} \widetilde {w_i'}$ (cond. 2). If we take the sum of these two conditions, then the second line of the obtained
 sum directly gives us $a_1=1$, $a_2=a_3=0$.

The position $(2,3)$ of the second condition gives $b_3(b_2+c_3)=0$, and since $b_2$ is invertible, we have $c_3=-b_2$.

Additionally we make the basis change with the help of the matrix
$$
\begin{pmatrix}
E_{i-2}& 0& 0& 0& 0\\
0& 1& 0& 0& 0\\
0& 0& 1& 0& 0\\
0& 0& 1+b_2& b_3& 0\\
0& 0& 0& 0& E_{l-i-1}
\end{pmatrix},
$$
the elements $\widetilde w_1$,\dots, $\widetilde w_{i-1}$ are not moved, but $\widetilde
{w_2'}$ is moved to
$$
\widetilde \varphi_1(w_1)=\begin{pmatrix} 1& 0& 0\\
1& -1& 1\\
x& y& 1
\end{pmatrix}.
$$
Now from the first condition it directly follows $x=y=0$, that is what we need.
\end{proof}

\begin{lemma}\label{l3_5}
For the root system $B_l$ if $\widetilde w_1=\widetilde{w_1'}$, \dots, $\widetilde w_{l-1}=\widetilde{w_{l-1}'}$,  then one can choose a basis in
$\widetilde V$ such that  $\widetilde w_1=\widetilde{w_1'}$, \dots, $\widetilde w_{l-1}=\widetilde{w_{l-1}'}$,
$\widetilde{w_l'}=\widetilde w_l$ in this basis.
\end{lemma}
\begin{proof}
Since $\widetilde{w_l'}$ commutes with all $\widetilde w_1$, \dots, $\widetilde w_{l-2}$ and has the order two, then on basis elements $e_1,\dots,e_{l-2}$ it is identical. Therefore we only need to consider last two basis elements, the matrix under consideration can be supposed a  $2\times 2$ matrix of the form
$$
\begin{pmatrix}
a& b\\
c& d
\end{pmatrix}.
$$

Since it has the order two, we get $c(a+d)=0$, and since $c\equiv 2\mod J$, then $a+d=0$.

After the basis change with the matrix
$$
\begin{pmatrix}
E_{l-2}& 0& 0\\
0& 1& (1-a)/c\\
0& 0 & 1+2(1-a)/c
\end{pmatrix}
$$
the elements $\widetilde w_1$, \dots, $\widetilde w_{l-1}$ are not moved, and $\widetilde{w_l'}$ becomes
$$
\begin{pmatrix}
E_{l-2}& 0& 0\\
0& 1& b\\
0& c& -1
\end{pmatrix}.
$$
Again from the fact that the matrix has the order two, and $c$ is invertible, we get $b=0$.

Now use the condition
$$
\widetilde w_{l-1}\widetilde
{w_l'} \widetilde w_{l-1}\widetilde{w_l'}=\widetilde {w_l'} \widetilde
w_{l-1} \widetilde {w_l'}\widetilde w_{l-1},
$$
it directly gives $c=2$.
\end{proof}

Therefore, we now can come from the isomorphism $\varphi_1$ under consideration to an isomorphism
 $\varphi_2$, with all properties of $\varphi_1$,
and such that $\varphi_2(w_i)=w_i$ for all $i=1,\dots,l$.

We suppose that we have the isomorphism~$\varphi_2$ with these
properties.

\section{Images of $x_{\alpha_i}(1)$ and $h_{\alpha_i}(t)$.}

In this section we want to prove that under one more basis change we can come to an isomorphism $\varphi_3$ with all properties of the
previous one and such that $\varphi_3(x_{\alpha}(1))=x_\alpha(1)$ for all $\alpha\in \Phi$, $\varphi_3(h_{\alpha}(t))=h_{\alpha}(s)$ for all $t\in R^*$, $\alpha\in \Phi$.

Unfortunately it is not possible to make a proof for all systems $B_l$, $l\geqslant 2$, simultaneously. We consider separately the cases $l=2$, $l=3$, $l=4$, $l\geqslant 5$.

Note that the case $B_2$ was studied in the paper~~\cite{korni2}.

\subsection{The case $B_3$.}

In this case we have the following roots: $\alpha_1=e_1-e_2$, $\alpha_2=e_2-e_3$, $\alpha_3=e_3$, $\alpha_4=\alpha_1+\alpha_2=e_1-e_3$, $\alpha_5=\alpha_2+\alpha_3=e_2$, $\alpha_6=\alpha_1+\alpha_2+\alpha_3=e_1$, $\alpha_7=\alpha_2+2\alpha_3=e_2+e_3$, $\alpha_8=\alpha_1+\alpha_2+2\alpha_3=e_1+e_3$, $\alpha_9=\alpha_1+2\alpha_2+2\alpha_3=e_1+e_2$.

Let $x_1=\varphi_2(x_{\alpha_1}(1))$, $x_3=\varphi_2(x_{\alpha_3}(1))$.

Note that
\begin{align*}
h_{\alpha_1}(-1)&=\diag[1,1,-1,-1,1,1,-1,-1,-1,-1,-1,-1,-1,-1, -1,-1, 1,1,1,1,1],\\
h_{\alpha_2}(-1)&=\diag[-1,-1,1,1,-1,-1,-1,-1,-1,-1,1,1,1,1,-1,-1,-1,-1,1,1,1].
\end{align*}

Recall for the convenience that

\begin{align*}
w_1&=-e_{\alpha_1,-\alpha_1}-e_{-\alpha_1,\alpha_1}+e_{\alpha_2,\alpha_4}+e_{-\alpha_2,-\alpha_4}-e_{\alpha_4,\alpha_2}
-e_{-\alpha_4,\alpha_2}+e_{\alpha_3,\alpha_3}+e_{-\alpha_3,-\alpha_3}+\\
&\ \ \ \ \  +e_{\alpha_5,\alpha_6}+e_{-\alpha_5,-\alpha_6}
-e_{\alpha_6,\alpha_5}-e_{-\alpha_6,\alpha_5}+e_{\alpha_7,\alpha_8}+e_{-\alpha_7,-\alpha_8}-e_{\alpha_8,\alpha_7}
-e_{-\alpha_8,-\alpha_7}+\\
&\ \ \ \ \ +e_{\alpha_9,\alpha_9}+e_{-\alpha_9,-\alpha_9}-e_{h_1,h_1}+e_{h_1,h_2}+e_{h_2,h_2}+e_{h_3,h_3};\\
w_2&=-e_{\alpha_2,-\alpha_2}-e_{-\alpha_2,\alpha_2}+e_{\alpha_1,\alpha_4}+e_{-\alpha_1,-\alpha_4}-e_{\alpha_4,\alpha_1}
-e_{-\alpha_4,\alpha_1}+\\
&\ \ \ \ \ +e_{\alpha_3,\alpha_5}+e_{-\alpha_3,-\alpha_5}
-e_{\alpha_5,\alpha_3}-e_{-\alpha_5,\alpha_3}+e_{\alpha_6,\alpha_6}+e_{-\alpha_6,-\alpha_6}
+e_{\alpha_7,\alpha_7}+e_{-\alpha_7,-\alpha_7}+\\
&\ \ \ \ \ +e_{\alpha_8,\alpha_9}+e_{-\alpha_8,-\alpha_9}-e_{\alpha_9,\alpha_8}
-e_{-\alpha_9,-\alpha_8}+
e_{h_1,h_1}+e_{h_2,h_1}-e_{h_2,h_2}+e_{h_2,h_3}+e_{h_3,h_3};\\
w_3&=e_{\alpha_1,\alpha_1}+e_{\alpha_1,\alpha_1}+e_{\alpha_2,\alpha_7}+e_{-\alpha_2,-\alpha_7}+e_{\alpha_7,\alpha_2}
+e_{-\alpha_7,\alpha_2}-e_{\alpha_3,-\alpha_3}-e_{-\alpha_3,\alpha_3}+\\
&\ \ \ \ \ +e_{\alpha_4,\alpha_8}+e_{-\alpha_4,-\alpha_8}
+e_{\alpha_8,\alpha_4}+e_{-\alpha_8,\alpha_4}+e_{\alpha_5,\alpha_5}+e_{-\alpha_5,-\alpha_5}+e_{\alpha_6,\alpha_6}
+e_{-\alpha_6,-\alpha_6}+\\
&\ \ \ \ \ +e_{\alpha_9,\alpha_9}+e_{-\alpha_9,-\alpha_9}+e_{h_1,h_1}+e_{h_2,h_2}+2e_{h_3,h_2}-e_{h_3,h_3};\\
X_1&= -e_{h_1,-\alpha_1}+2e_{\alpha_1,h_1}-e_{\alpha_1,h_2}+e_{\alpha_4,\alpha_2}-e_{-\alpha_2,-\alpha_4}+
e_{\alpha_6,\alpha_5}-e_{-\alpha_5,-\alpha_6}+\\
&\ \ \ \ \ +e_{\alpha_8,\alpha_7}-e_{-\alpha_7,-\alpha_8},\\
X_3&=-e_{h_3,-\alpha_3}+2e_{\alpha_3,h_3}-2e_{\alpha_3,h_2}+2e_{\alpha_5,\alpha_2}-e_{-\alpha_2,-\alpha_5}
+2e_{\alpha_6,\alpha_4}-e_{-\alpha_4,-\alpha_6}-\\
&\ \ \ \ \ -e_{\alpha_7,\alpha_5}+2e_{-\alpha_5,-\alpha_7}
-e_{\alpha_8,\alpha_6}+2e_{-\alpha_6,-\alpha_8}.
\end{align*}

Since the matrices $x_1$ and $x_3$ commute with $h_{\alpha_1}(-1)$, then both of them are separated into the blocks corresponding to the basis parts $\{ v_{1}, v_{-1}, v_{3}, v_{-3},  v_{9}, v_{-9}, V_1, V_2, V_3\}$ and $\{ v_2, v_{-2}, v_4, v_{-4}, v_5, v_{-5}$, $v_6, v_{-6}, v_7, v_{-7}, v_8, v_{-8}\}$.

Since the matrix $x_1$ commutes with $w_3$ and $w_2w_3w_2w_1w_2^{-1}w_3^{-1}w_2^{-1}$, it follows that on the first block it has the form
{\tiny
$$
\begin{pmatrix}
t_1& t_2& t_3& -t_3& t_4& -t_4& -2t_5& t_5& 0\\
t_6& t_7& t_8& -t_8& t_9& -t_9& -2t_{10}& t_{10}& 0\\
t_{11}& t_{12}& t_{13}& t_{14}& -t_{15}& t_{15}& -2t_{16}-2t_{17}& t_{16}& t_{17}\\
-t_{11}& -t_{12}& t_{14}& t_{13}& t_{15}& -t_{15}& 2t_{16}+2t_{17}& t_{16}-2t_{17}& t_{17}\\
t_{18}& t_{19}& -t_{20}& t_{20}& t_{21}& t_{22}& t_{23}& t_{24}& 0\\
-t_{18}& -t_{19}& t_{20}& -t_{20}& t_{22}& t_{21}& -t_{23}& t_{23}+t_{24}& 0\\
t_{25}& t_{26}& -t_{27}& t_{27}& t_{28}& t_{29}& t_{30}& t_{31}& 0\\
0& 0& 0& 0& t_{28}+t_{29}& t_{28}+t_{29}& 0& 2t_{31}+t_{30}& 0\\
0& 0& t_{32}& t_{32}& t_{28}+t_{29}& t_{28}+t_{29}& 0& t_{33}& 2t_{31}-t_{33}+t_{30}
\end{pmatrix};
$$
}
and on the second block it is
{\tiny
$$
\left(\begin{array}{cccccccccccc}
t_{34}& t_{35}& t_{36}& t_{37}& t_{38}& t_{39}& t_{40}& t_{41}& t_{42}& t_{43}& t_{44}& t_{45}\\
t_{46}& t_{47}& t_{48}& t_{49}& t_{50}& t_{51}& t_{52}& t_{53}& t_{54}& t_{55}& t_{56}& t_{57}\\
-t_{49}& -t_{48}& t_{47}& t_{46}& -t_{53}& -t_{52}& t_{51}& t_{50}& -t_{57}& -t_{56}& t_{55}& t_{54}\\
-t_{37}& -t_{36}& t_{35}& t_{34}& -t_{41}& -t_{40}& t_{39}& t_{38}& -t_{45}& -t_{44}& t_{43}& t_{42}\\
t_{58}& t_{59}& t_{60}& t_{61}& t_{62}& t_{63}& t_{64}& t_{65}& t_{58}& t_{59}& t_{60}& t_{61}\\
t_{66}& t_{67}& t_{68}& t_{69}& t_{70}& t_{71}& t_{72}& t_{73}& t_{66}& t_{67}& t_{68}& t_{69}\\
-t_{69}& -t_{68}& t_{67}& t_{66}& -t_{73}& -t_{72}& t_{71} & t_{70}& -t_{69}& -t_{68}& t_{67}& t_{66}\\
-t_{61}& -t_{60}& t_{59}& t_{58}& -t_{65}& -t_{64}& t_{63}& t_{62}& -t_{61}& -t_{60}& t_{59}& t_{58}\\
t_{42}& t_{43}& t_{44}& t_{45}& t_{38}& t_{39}& t_{40}& t_{41}&t_{34}& t_{35}& t_{36}& t_{37}\\
t_{54}& t_{55}& t_{56}& t_{57}& t_{38}& t_{39}& t_{40}& t_{41}&t_{34}& t_{35}& t_{36}& t_{37}\\
-t_{57}& -t_{56}& t_{55}& t_{54}& -t_{53}& -t_{52}& t_{51}& t_{50}&-t_{49}& -t_{48}& t_{47}& t_{46}\\
-t_{45}& -t_{44}& t_{43}& t_{42}& -t_{41}& -t_{40}& t_{39}& t_{38}&-t_{37}& -t_{36}& t_{35}& t_{34}
\end{array}\right).
$$
}

Similarly, since the matrix $x_3$ commutes with $w_1$ and $w_2w_3w_2^{-1}$, we have that on the first block it is
{\tiny
$$
\begin{pmatrix}
u_1& u_2& u_3& u_4& u_5& -u_5& -2u_6-2u_7& u_6& u_7\\
u_2& u_1& -u_3& -u_4& -u_5& u_5& -2u_6-2u_7& u_6+2u_7& -u_7\\
-u_8& u_8& u_9& u_{10}& u_8& -u_8& 0& -u_{11}& u_{11}\\
u_{12}& -u_{12}& u_{13}& u_{14}& -u_{12}& u_{12}& 0&  -u_{15}& u_{15}\\
u_5& -u_5& -u_3& -u_4& u_1& u_2& 0& u_6+2u_7& -u_7\\
-u_5& u_5& u_3& u_4& u_2& u_1& 0& u_6& u_7\\
-u_{16}& -u_{16}& 0& 0& u_{16}& u_{16}& u_{20}+u_{21}&  0& 0\\
0& 0& 0& 0& 2u_{16}& 2u_{16}& 0& u_{20}+u_{21}& 0\\
-u_{17}& u_{17}& u_{18}& u_{19}& 2u_{16}+u_{17}& 2u_{16}+u_{17}& u_{20}& u_{21}
\end{pmatrix};
$$
}
and on the second block it is
{\tiny
$$
\left(\begin{array}{cccccccccccc}
u_{22}& u_{23}& -u_{24}& -u_{25}& u_{27}& u_{26}& -u_{28}& u_{28}& u_{29}& u_{30}& u_{25}& u_{24}\\
u_{31}& u_{32}& -u_{33}& -u_{34}& u_{36}& u_{35}& -u_{37}& u_{37}& u_{38}& u_{39}& u_{34}& u_{33}\\
u_{24}& u_{25}& u_{22}& u_{23}& u_{28}& -u_{28}& u_{27}& u_{26}& -u_{25}& -u_{24}& u_{29}& u_{30}\\
u_{33}& u_{34}& u_{31}& u_{32}& u_{37}& -u_{37}& u_{36}& u_{35}& -u_{34}& -u_{33}& u_{38}& u_{39}\\
u_{40}& u_{41}& -u_{42}& -u_{43}& u_{44}& u_{45}& -u_{46}& u_{46}& u_{47}& u_{48}& u_{43}& u_{42}\\
u_{48}& u_{47}& u_{42}& u_{43}& u_{45}& u_{44}& u_{46}& -u_{46}& u_{41}& u_{40}& -u_{43}& -u_{42}\\
u_{42}& u_{43}& u_{40}& u_{41}& u_{46}& -u_{46}& u_{44}& u_{45}& -u_{43}& -u_{42}& u_{47}& u_{48}\\
-u_{42}& -u_{43}& u_{48}& u_{47}& -u_{46}& u_{46}& u_{45}& u_{44}& u_{43}& u_{42}& u_{41}& u_{40}\\
u_{39}& u_{38}& u_{33}& u_{34}& u_{35}& u_{36}& u_{37}& -u_{37}& u_{32}& u_{31}& -u_{34}& -u_{33}\\
u_{30}& u_{29}& u_{24}& u_{25}& u_{26}& u_{27}& u_{28}& -u_{28}& u_{23}& u_{22}& -u_{25}& -u_{24}\\
-u_{33}& -u_{34}& u_{39}& u_{38}& -u_{37}& u_{37}& u_{35}& u_{36}& u_{34}& u_{33}& u_{32}& u_{31}\\
-u_{24}& -u_{25}& u_{30}& u_{29}& -u_{28}& u_{28}& u_{26}& u_{27}& u_{25}& u_{24}& u_{23}& u_{22}
 \end{array}\right).
$$
}

We have that $t_1,t_7,t_{13},t_{21},t_{30},t_{34},t_{47},t_{62},t_{71},u_1,u_9,u_{14},u_{21},u_{22},u_{32},u_{44}$ are equivalent to~$1$ modulo radical, $t_2,t_5,t_{26},t_{49},t_{73},u_{10},u_{19},u_{35},u_{39}$ are $-1$ modulo radical, $u_{11},u_{40}$ are $2$ modulo radical, all other elements are form the radical. There are $121$ variables $t_1,\dots, t_{73}, u_1,\dots, u_{48}$.

We apply step by step three basis changes, commuting with each other and with all matrices $w_{i}$. These changes are represented by matrices $C_1$, $C_2$, $C_3$. The matrix $C_1$  is block-diagonal with  $2\times 2$ blocks. On all $2\times 2$ blocks, corresponding to long and short roots, the matrix  $C_1$  has the form
$$
\begin{pmatrix}
1& -u_{18}/u_{19}\\
-u_{18}/u_{19}& 1
\end{pmatrix}.
$$

On the last block it is identical.

The matrix $C_2$  is diagonal, it is identical on the last block,  it is scalar with~$a$ on all places corresponding to the roots.

The matrix $C_3$ is the sum of an identical matrix and  $2\times 2$ blocks on the following places: the lines correspond to the basis part $\{ v_\alpha, v_{-\alpha}\}$, $\alpha=e_1-e_2, e_2+e_3, e_1-e_3$, the rows correspond to the basis part $\{ v_\beta,v_{-\beta}\}$, $\beta=e_3, e_1,e_2$, and the matrix
$$
\begin{pmatrix}
t_3& -t_3\\
-t_3& t_3
\end{pmatrix};
$$
if the lines correspond to the basis part $\{ v_\alpha, v_{-\alpha}\}$, $\alpha=e_1+e_2, e_2-e_3, e_1+e_3$, then the rows again correspond to the basis part $\{ v_\beta,v_{-\beta}\}$, $\beta=e_3, e_1,e_2$, and the matrix is opposite.

Since all three matrices commute with all $w_{i}$,  then after the basis change with any of these matrices all conditions on the elements of the matrices $x_1$ and $x_3$ still hold.

At first we apply the basis change with $C_1$. After it new  $u_{18}$  in $x_3$ becomes equal to zero (for the convenience of notations we do not rename variables). Then choose  $a=-1/u_{19}$ (where $u_{19}$ is the last one). We see that $u_{18}$ is not moved and $u_{19}$ becomes equal to $-1$. Finally, applying the last change we get $t_3=0$.

Now we can suppose that $u_{18}=0, u_{19}=-1, t_{3}=0$, and we have $118$ variables.

Introduce $x_{4}=\varphi_2(x_{\alpha_4}(1))=w_2x_1{w_2^{-1}}$, $x_{2}=\varphi_2(x_{\alpha_2}(1))=w_1x_{4}{w_1^{-1}}$, $x_{7}=\varphi_2(x_{\alpha_7}(1))=w_3x_2{w_3^{-1}}$, $x_{5}=\varphi_2(x_{\alpha_5}(1))=w_2x_{3}{w_2^{-1}}$.

Use now the following conditions, that are true for the elements $w_i$ and $x_i$:
\begin{align*}
Con1&=(h_2x_1h_2x_1=E);\\
Con2&=(x_1x_{4}=x_{4}x_1);\\
Con3&=(x_1x_{3}=x_{3}x_1);\\
Con4&=(x_1x_2=x_4x_2x_1);\\
Con5&=(x_7x_{3}=x_{3}x_7);\\
Con6&=(h_2x_3h_2x_3=E);\\
Con7&=(x_7^2x_3x_{5}=x_{5}x_3).
\end{align*}

Note that every matrix condition is the set of $441$
polynomial identities, where all polynomials have integer coefficients and depend of the variables $t_i, u_j$. Temporarily rename all variables in $v_1,\dots,v_{118}$.

Suppose that one of our polynomials can be rewritten in the form
\begin{multline*}
(v_{k_0}-\overline v_{k_0})A+(v_1-\overline v_{1})B_1+\dots\\
\dots+(v_{k_0-1}-\overline v_{k_0-1})B_{k_0-1}+(v_{k_0+1}-\overline v_{k_0+1})B_{k_0+1}+\dots+v_{118}-\overline v_{118})B_{118}=0,
\end{multline*}
where $\overline  v_{i}$ is such an integer number that is equivalent to $v_{i}$ modulo radical, the polynomial $A$ in invertible modulo radical, $B_i$ are some polynomials (the variable $v_{k_0}$ can enter in all polynomial and also in~$A$). Then
\begin{multline*}
v_{k_0}-\overline v_{k_0}=\\
=-\frac{(v_1-\overline v_{1})B_1+\dots+(v_{k_0-1}-\overline v_{k_0-1})B_{k_0-1}+(v_{k_0+1}-\overline v_{k_0+1})B_{k_0+1}+\dots+v_{118}-\overline v_{118})B_{118}}{A},
\end{multline*}
we can substitute  the expression for $v_{k_0}$ in all other polynomial conditions. If we can choose $118$ such conditions that on every step we except one new variable, then on the last step we obtain the expression
$$
(v_{k_{118}}-\overline v_{k_{118}})C=0,
$$
where $C$ is some rational expression of variables
$v_1,\dots, v_{118}$, invertible modulo radical. Therefore,
we can say that $v_{k_{118}}=\overline v_{k_{118}}$, and consequently all other variables are equal to the integer numbers equivalent them modulo radical. The existence of the obtained $118$
conditions is equivalent to the existence of such $118$ conditions that
the square matrix consisting of all coefficients of these conditions modulo radical has an invertible determinant.

Since it is very complicated to write a matrix $118\times 118$, we will  sequentially take the obtained equations, but for  simplicity write coefficients $A$ and $B_i$ modulo radical (in the result these coefficients are just numbers $0$,  $\pm 1$, $\pm 2$).

We write below how the variables are expressed from the conditions (in brackets we write the number of the condition and the position there): $(Con1, 1,1)$: $t_6=2(t_1-1)$; $(Con1, 1,2)$: $t_7=-t_1-2t_5$; $(Con1, 1,5)$: $t_8=2t_{27}$; $(Con1, 1,17)$: $t_{28}=2t_4-t_9+t_{29}$; $(Con1, 1,19)$: $t_{10}=t_{30}-t_1$;  $(Con1, 5,1)$: $t_{11}=0$;  $(Con1, 5,2)$: $t_{12}=2t_{16}+2t_{17}$; $(Con1, 5,5)$: $t_{13}=1$; $(Con1, 5,6)$: $t_{14}=0$; $(Con1, 5,17)$: $t_{15}=0$; $(Con1, 17,1)$: $t_{18}=0$; $(Con1, 17,2)$: $t_{23}=-t_{19}$; $(Con1, 17,5)$: $t_{20}=0$; $(Con1, 19,5)$: $t_{27}=0$; $(Con1, 19,1)$: $t_{1}=1$; $(Con1, 19,20)$: $t_{30}=1-2t_{31}$; $(Con1, 19,18)$: $t_{9}=0$; $(Con1, 2,2)$: $t_{31}=1+t_5$;
$(Con1, 1,20)$: $t_{5}=-1$; $(Con1, 18,18)$: $t_{21}=1$; $(Con1, 6,2)$: $t_{17}=-t_{16}$; $(Con1, 18,17)$: $t_{22}=0$; $(Con1, 21,21)$: $t_{33}=0$; $(Con1, 18,2)$: $t_{19}=0$; $(Con1, 2,1)$: $t_{2}=-1$; $(Con2, 2,2)$: $t_{35}=0$; $(Con2, 5,2)$: $t_{60}=0$; $(Con2, 1,19)$: $t_{47}=1$; $(Con2, 6,2)$: $t_{68}=0$; $(Con2, 7,15)$: $t_{43}=0$;
$(Con2, 17,2)$: $t_{55}=0$; $(Con2, 19,18)$: $t_{42}=t_{29}$; $(Con2, 7,8)$: $t_{34}=1$; $(Con1, 3,3)$: $t_{36}=0$; $(Con1, 3,4)$: $t_{37}=0$; $(Con1, 4,3)$: $t_{48}=2t_{46}$; $(Con2, 1,8)$: $t_{49}=-1$; $(Con2, 4,5)$: $t_{52}=0$; $(Con2, 9,2)$: $t_{58}=t_{16}$; $(Con2, 9,5)$: $t_{64}=0$; $(Con2, 15,4)$: $t_{29}=0$; $(Con2, 20,3)$: $t_{25}=0$; $(Con1, 10,3)$: $t_{61}=0$; $(Con1, 10,7)$: $t_{59}=0$; $(Con2, 1,11)$: $t_{39}=0$; $(Con2, 4,8)$: $t_{46}=0$; $(Con2, 16,2)$ $t_{45}=t_{24}$; $(Con2, 7,2)$: $t_{26}=-1$;
$(Con1, 4,9)$: $t_{41}=0$; $(Con1, 9,12)$: $t_{63}=0$; $(Con2, 4,9)$: $t_{40}=0$; $(Con1, 13,8)$: $t_{44}=0$; $(Con3, 17,2)$: $u_{5}=0$; $(Con3, 2,19)$: $u_{2}=0$; $(Con3, 5,19)$: $u_{8}=0$; $(Con3, 6,19)$: $u_{12}=0$; $(Con3, 5,5)$: $t_{32}=0$; $(Con3, 5,6)$: $t_{16}=0$; $(Con3, 21,2)$: $u_{17}=0$; $(Con3, 1,1)$: $u_{16}=0$; $(Con2, 14,2)$: $t_{56}=-t_{54}$; $(Con2, 14,19)$ and $(Con2, 14,21)$: $t_{54}=t_{24}=0$; $(Con3, 19,19)$: $u_{7}=-u_6$; $(Con4, 5,5)$: $t_{62}=1$; $(Con4, 5,10)$: $t_{65}=0$; $(Con4, 6,21)$: $t_{67}=0$; $(Con4, 9,20)$: $t_{66}=0$; $(Con4, 21,13)$: $t_{4}=0$; $(Con4, 2,13)$: $t_{57}=0$; $(Con4, 6,18)$: $t_{69}=0$; $(Con4, 4,5)$: $t_{50}=0$; $(Con4, 6,6)$: $t_{71}=1$;
$(Con4, 6,5)$: $t_{70}=0$; $(Con4, 17,5)$: $t_{53}=0$; $(Con4, 11,10)$: $t_{72}=0$; $(Con4, 11,5)$: $t_{73}=-1$; $(Con4, 3,5)$: $t_{38}=0$; $(Con4, 3,6)$: $t_{51}=0$; $(Con3, 9,3)$: $u_{42}=0$; $(Con3, 16,3)$: $u_{30}=0$; $(Con3, 12,3)$: $u_{48}=0$; $(Con3, 11,4)$: $u_{41}=0$; $(Con3, 10,4)$: $u_{43}=0$; $(Con3, 4,7)$: $u_{31}=0$; $(Con3, 3,8)$: $u_{23}=0$; $(Con3, 9,12)$: $u_{45}=0$; $(Con3, 8,13)$: $u_{38}=0$; $(Con3, 9,3)$: $u_{1}=u_{20}+u_{21}$; $(Con3, 14,4)$: $u_{25}=0$; $(Con3, 4,4)$: $u_{34}=0$; $(Con3, 4,11)$: $u_{36}=0$; $(Con3, 7,7)$: $u_{24}=0$; $(Con3, 1,20)$: $u_{6}=0$;
$(Con5, 2,8)$: $u_{29}=0$; $(Con5, 9,5)$: $u_{13}=0$; $(Con5, 5,4)$: $u_{47}=0$; $(Con5, 3,6)$: $u_{27}=0$; $(Con5, 9,20)$: $u_{15}=0$; $(Con5, 13,20)$: $u_{20}=0$; $(Con5, 13,19)$: $u_{32}=u_{21}$; $(Con6, 8,8)$: $u_{21}=1$; $(Con6, 17,20)$: $u_{3}=0$; $(Con6, 5,20)$: $u_{9}=1$; $(Con6, 12,9)$: $u_{28}=0$; $(Con6, 6,6)$: $u_{14}=1$; $(Con6, 2,6)$: $u_{4}=0$; $(Con6, 3,3)$: $u_{22}=1$; $(Con6, 10,14)$: $u_{44}=1$; $(Con5, 5,14)$: $u_{40}=u_{11}$; $(Con5, 13,9)$: $u_{26}=0$; $(Con6, 5,6)$: $u_{11}=-2u_{10}$; $(Con6, 12,3)$: $u_{46}=0$; $(Con7, 13,14)$: $u_{10}=-1$; $(Con7, 17,5)$: $u_{37}=0$; $(Con7, 5,7)$: $u_{33}=0$;  $(Con7, 5,14)$: $u_{39}=-1$; $(Con7, 2,16)$: $u_{35}=-1$.

Therefore $x_1=x_{\alpha_1}(1)$ and $x_3=x_{\alpha_3}(1)$. Since all long (and all short) roots are conjugate under the action of the Weil group, we have $\varphi_2(x_{\alpha}(1))=x_\alpha(1)$ for all $\alpha\in \Phi$.

Now let us consider $h_t=\varphi_2(h_{\alpha_1}(t))$. Since $h_t$ commutes with $h_1$, $h_2$, $w_3$, $w_9$, $x_{\alpha_3}(1)$,  $x_{\alpha_9}(1)$, and also $w_1h_tw_1^{-1}h_t=E$, $w_2h_tw_2^{-1}=w_1w_2h_tw_2^{-1}w_1^{-1}h_t$, we directly have $h_t=h_{\alpha_1}(s)$ for some $s\in R^*$.
Similarly, $\varphi_2(h_{\alpha_3}(t))=h_{\alpha_3}(s)$.

\subsection{The case $B_4$.}

In this case we have $32$ roots
$$
\pm e_1,\pm e_2,\pm e_3, \pm e_4, \pm e_i \pm e_j,\qquad 1\leqslant i< j\leqslant 4,
 $$
the adjoint representation is $36$-dimensional. Similarly to the previous case we consider $x_1=\varphi_2(x_{\alpha_1}(1))$ and $x_4=\varphi_2(x_{\alpha_4}(1))$. According to commuting with the elements $h_1$ and $h_3$ the matrix $x_1$ is separated into the blocks
 \begin{align*}
&\{ v_{e_1-e_2}, v_{e_2-e_1}, v_{e_1+e_2}, v_{-e_1-e_2},  v_{e_3-e_4}, v_{e_4-e_3}, v_{e_3+e_4}, v_{-e_3-e_4}, V_1, V_2, V_3, V_4\},\\
&\{ v_{e_1-e_3}, v_{e_3-e_1}, v_{e_1+e_3}, v_{-e_1-e_3},  v_{e_1-e_4}, v_{e_4-e_1}, v_{e_1+e_4}, v_{-e_1-e_4},\\
& \quad \quad \quad \quad \quad v_{e_2-e_3}, v_{e_3-e_2}, v_{e_2+e_3}, v_{-e_2-e_3}, v_{e_2-e_4}, v_{e_4-e_2}, v_{e_2+e_4}, v_{-e_2-e_4}\},\\
&\{ v_{e_1}, v_{-e_1}, v_{e_2}, v_{-e_2}\},\\
& \{ v_{e_3}, v_{-e_3}, v_{e_4}, v_{-e_4}\},
 \end{align*}
according to commuting with $h_1$ and $h_2$ the matrix $x_4$ is separated into the blocks
\begin{align*}
& \{ v_{e_4}, v_{-e_4}, V_1, V_2, V_3, V_4\},\\
& \{ v_{e_1-e_2}, v_{e_2-e_1}, v_{e_1+e_2}, v_{-e_1-e_2}, v_{e_3}, v_{-e_3}, v_{e_3-e_4}, v_{e_4-e_3}, v_{e_3+e_4}, v_{-e_3-e_4}\},\\
&\{ v_{e_1}, v_{-e_1}, v_{e_1-e_4}, v_{e_4-e_1}, v_{e_1+e_4}, v_{-e_1-e_4}, v_{e_2-e_3}, v_{e_3-e_2}, v_{e_2+e_3}, v_{-e_2-e_3}\},\\
&\{ v_{e_2}, v_{-e_2}, v_{e_2-e_4}, v_{e_4-e_2}, v_{e_2+e_4}, v_{-e_2-e_4}, v_{e_1-e_3}, v_{e_3-e_1}, v_{e_1+e_3}, v_{-e_1-e_3}\}.
\end{align*}

At the beginning we consider the matrix $x_1$ on the block $\{ v_{e_3}, v_{-e_3}, v_{e_4}, v_{-e_4}\}$, its inverse image on this block is just identical. We know that $x_1$ commutes with $w_{e_3}$, $w_{e_4}$, $w_{e_3-e_4}$. It gives us the following form of the matrix on the block under consideration:
$$
\begin{pmatrix}
a_1& a_2&  a_3& -a_3\\
a_2& a_1& -a_3& a_3\\
-a_3& a_3& a_1& a_2\\
a_3& -a_3& a_2& a_1
\end{pmatrix}.
$$
Besides, since $h_2x_1h_2x_1=E$, we have the equations $a_1^2+a_2^2+2a_3^2=1$ and $2a_1a_2-2a_3^2=0$. Their sum is $(a_1+a_2)^2=1$, therefore (since $a_1+a_2\equiv 1\mod  J$) we have $a_1+a_2=1$, i.\,e., $a_1=1-a_2$.

Now consider our matrix on the block $\{ v_{e_1}, v_{-e_1}, v_{e_2}, v_{-e_2}\}$. Its inverse image on this block is
$$
\begin{pmatrix}
1& 0& -1& 0\\
0& 1& 0& 0\\
0& 0& 1& 0\\
0& 1& 0& 1
\end{pmatrix}.
$$
Since it commutes with $w_{e_1+e_2}$ we have the form
$$
\begin{pmatrix}
b_1& b_2& b_3& b_4\\
c_1& c_2& c_3& c_4\\
-c_4& -c_3& c_2& c_1\\
-b_4& -b_3& b_2& b_1
\end{pmatrix}.
$$

Now unite these bases together, we can note that for all $w_{\alpha}$, $\alpha\in \Phi$, and all $x_{\beta}(1)$, $\beta$ is a long root, this basis part is an invariant  direct summand. In the matrix under consideration there are $11$ variables. Consider on this basis part the matrix $\varphi_2(x_{e_1-e_3}(1))=w_{e_2-e_3}x_1w_{e_2-e_3}^{-1}$. It is
$$
\begin{pmatrix}
b_1& b_2& 0& 0& -b_3& -b_4& 0& 0\\
c_1& c_2& 0& 0& -c_3& -c_4& 0& 0\\
0& 0& 1-a_2& a_2& 0& 0& a_3& -a_3\\
0& 0& a_2& 1-a_2& 0& 0& -a_3& a_3\\
c_4& c_3& 0& 0& c_2& c_1& 0& 0\\
b_4& b_3& 0& 0& b_2& b_1& 0& 0\\
0& 0& -a_3& a_3& 0& 0& 1-a_2& a_2\\
0& 0& a_3& -a_3& 0& 0& a_2& 1-a_2
\end{pmatrix}
$$

We know that the matrices
$\varphi_2(x_{e_1-e_3}(1))$ and $x_1$ commute, therefore (the position $(1,8)$) it follows $2a_3(b_3-b_4)=0$. Since $b_3-b_4\equiv 1\mod J$, we have $a_3=0$. From the equations listed above we have $2a_2(a_2-1)=0$, therefore (since $a_2\in J$) $a_2=0$. Now use the condition $x_{e_1-e_2}(1)x_{e_2-e_3}(1)=x_{e_1-e_3}(1)x_{e_2-e_3}(1)x_{e_1-e_2}(1)$.
The position $(6,5)$ gives us $b_2(1-b_1-c_2)$, so $b_2=0$. After it the position $(6,6)$ gives $b_1(b_1-1)=0$, therefore а $b_1=1$. From the position $(5,5)$ we have $c_2=1$, and from the position $(5,6)$ we have $c_1=0$. From the same condition  $b_4=-c_3$, $c_4=-b_4^2/b_3$. Therefore the matrix under consideration on the given basis part is
$$
\begin{pmatrix}
1& 0& b_3& -c_3& 0& 0& 0& 0\\
0& 1& c_3& -c_3^2/b_3& 0& 0& 0& 0\\
c_3^2/b_3& -c_3& 1& 0& 0& 0& 0& 0\\
c_3& -b_3& 0& 1& 0& 0& 0& 0\\
0& 0& 0& 0& 1& 0& 0& 0\\
0& 0& 0& 0& 0& 1& 0& 0\\
0& 0& 0& 0& 0& 0& 1& 0\\
0& 0& 0& 0& 0& 0& 0& 1
\end{pmatrix}.
$$

Now take another basis part:
$$\{ v_{e_1-e_2}, v_{e_2-e_1}, v_{e_1+e_2}, v_{-e_1-e_2},  v_{e_3-e_4}, v_{e_4-e_3}, v_{e_3+e_4}, v_{-e_3-e_4}, v_{e_3}, v_{-e_3}, v_{e_4}, v_{-e_4}, V_1, V_2, V_3, V_4\}.
 $$
The matrix $x_1$ commutes with $w_{e_1+e_2}$, $w_{e_4}$, $w_{e_3-e_4}$, therefore we directly have that on the unknown  $12\times 12$ basis part the matrix is
{\footnotesize
\begin{multline*}
\left(\begin{matrix}
b_{1,1}& b_{1,2}& b_{1,3}& -b_{1,3}& b_{1,5}& -b_{1,5}& b_{1,5}& -b_{1,5}\\
b_{2,1}& b_{2,2}& b_{2,3}& -b_{2,3}& b_{2,5}& -b_{2,5}& b_{2,5}& -b_{2,5}\\
b_{3,1}& b_{3,2}& b_{3,3}& b_{3,4}& b_{3,5}& -b_{3,5}& b_{3,5}& -b_{3,5}\\
-b_{3,1}& -b_{3,2}& b_{3,4}& b_{3,3}& -b_{3,5}& b_{3,5}& -b_{3,5}& b_{3,5}\\
b_{5,1}& b_{5,2}& b_{5,3}& -b_{5,3}& b_{5,5}& b_{5,6}& b_{5,7}& -b_{5,7}\\
-b_{5,1}& -b_{5,2}& -b_{5,3}& b_{5,3}& b_{5,6}& b_{5,5}& -b_{5,7}& b_{5,7}\\
b_{5,1}& b_{5,2}& b_{5,3}& -b_{5,3}& b_{5,7}& -b_{5,7}& b_{5,5}& b_{5,6}\\
-b_{5,1}& -b_{5,2}& -b_{5,3}& b_{5,3}& -b_{5,7}& b_{5,7}& b_{5,6}& b_{5,5}\\
b_{13,1}& b_{13,2}& b_{13,3}& b_{15,3}-b_{13,3}& b_{13,5}& -b_{13,5}& b_{13,5}& -b_{13,5}\\
0& 0& b_{15,3}& b_{15,3}& 0& 0& 0& 0\\
0& 0& b_{15,3}& b_{15,3}& b_{15,5}& b_{15,5}& b_{15,5}& b_{15,5}\\
0& 0& b_{15,3}& b_{15,3}& 0& 0& 2b_{15,5}& 2b_{15,5}
\end{matrix}\right.\\
\left.\begin{matrix}
-2b_{1,14}& b_{1,14}& 0& 0\\
-2b_{2,14}& b_{2,14}& 0& 0\\
 b_{4,14}-b_{3,14}& b_{3,14}& 0& 0\\
 b_{3,14}-b_{4,14}& b_{4,14}& 0& 0\\
 2(b_{8,14}-b_{6,16})& 2b_{6,16}-b_{8,14}& -2b_{6,16}& b_{6,16}\\
 2(-b_{8,14}+b_{6,16})& b_{8,14}& -2b_{6,16}& b_{6,16}\\
 2(b_{8,14}-b_{6,16})& 2b_{6,16}-b_{8,14}& 0& -b_{6,16}\\
 2(-b_{8,14}+b_{6,16})& b_{8,14}& 0& -b_{6,16}\\
 b_{14,14}-2b_{13,14}& b_{13,14}& 0& 0\\
  0& b_{14,14}& 0& 0\\
 0& b_{8,14}& b_{14,14}-b_{8,14}& 0\\
 0& b_{8,14}& 0& b_{14,14}-b_{8,14}
\end{matrix}\right)
\end{multline*}
}

Since $x_4$ commutes with $w_{e_1-e_2}$, $w_{e_1+e_2}$, $w_{e_3}$, we have that on the basis part
$$
\{ v_{e_1-e_2}, v_{e_2-e_1}, v_{e_1+e_2}, v_{-e_1-e_2}, v_{e_3-e_4}, v_{e_4-e_3}, v_{e_3+e_4}, v_{-e_3-e_4}, v_{e_3}, v_{-e_3}\}
$$
it is
$$
\begin{pmatrix}
a_{1,1}& a_{1,2}& a_{1,3}& -a_{1,3}& a_{1,5}& a_{1,6}& -a_{1,6}& -a_{1,5}& a_{1,9}& -a_{1,9}\\
a_{1,2}& a_{1,1}& -a_{1,3}& a_{1,3}& -a_{1,5}& -a_{1,6}& a_{1,6}& a_{1,5}& -a_{1,9}& a_{1,9}\\
a_{3,1}& -a_{3,1}& a_{3,3}& a_{3,4}& a_{3,5}& a_{3,6}& -a_{3,6}& -a_{3,5}& a_{3,9}& -a_{3,9}\\
-a_{3,1}& a_{3,1}& a_{3,4}& a_{3,3}& -a_{3,5}& -a_{3,6}& a_{3,6}& a_{3,5}& -a_{3,9}& a_{3,9}\\
a_{5,1}& -a_{5,1}& a_{5,3}& -a_{5,3}& a_{5,5}& a_{5,6}& a_{5,7}& a_{5,8}& a_{5,9}& a_{5,10}\\
a_{6,1}& -a_{6,1}& a_{6,3}& -a_{6,3}& a_{6,5}& a_{6,6}& a_{6,7}& a_{6,8}& a_{6,9}& a_{6,10}\\
-a_{6,1}& a_{6,1}& -a_{6,3}& a_{6,3}& a_{6,8}& a_{6,7}& a_{6,6}& a_{6,5}& a_{6,10}& a_{6,9}\\
-a_{5,1}& a_{5,1}& -a_{5,3}& a_{5,3}& a_{5,8}& a_{5,7}& a_{5,6}& a_{5,5}& a_{5,10}& a_{5,9}\\
a_{9,1}& -a_{9,1}& a_{9,3}& -a_{9,3}& a_{9,5}& a_{9,6}& a_{9,7}& a_{9,8}& a_{9,9}& a_{9,10}\\
-a_{9,1}& a_{9,1}& -a_{9,3}& a_{9,3}& a_{9,8}& a_{9,7}& a_{9,6}& a_{9,5}& a_{9,10}& a_{9,9}
\end{pmatrix},
$$
and on the basis part $\{ v_{e_4}, v_{-e_4}, V_1, V_2, V_3, V_4\}$ (according to $h_{e_3-e_4}(-1)x_4h_{e_3-e_4}(-1)x_4=E$) it is
$$
\begin{pmatrix}
a_{11,11}& a_{11,12}& 0& 0& a_{11,15}& -a_{11,15}\\
a_{12,11}& a_{12,12}& 0& 0& a_{12,15}& -a_{12,15}\\
0& 0& 1& 0& 0& 0\\
0& 0& 0& 1& 0& 0\\
0& 0& 0& 0& 1& 0\\
a_{16,11}& a_{16,12}& 0& 0& 1-a_{16,16}& a_{16,16}
\end{pmatrix}.
$$

Similarly to the previous section with some block-diagonal basis change we can make $a_{16,11}=0$, $a_{16,12}=-1$.

From the position $(15,14)$ of
$$
Con1=(h_{e_3-e_4}(-1)x_4h_{e_3-e_4}(-1)x_4=E)
$$
it follows $a_{16,14}(a_{16,14}+2)=0$, so $a_{16,14}=0$. Similarly, from the position $(13,14)$ we have $a_{13,14}=0$, from the position $(16,12)$ $a_{16,16}=a_{12,12}$, from $(16,11)$  $a_{12,11}=0$, from $(11,11)$ $a_{11,11}=1$; from $(16,16)$  $a_{12,15}=-1+a_{12,12}^2$ from $(11,12)$  $a_{11,15}=a_{11,12}(1+a_{12,12})$.

Since $x_1$ and $x_4$ commute, we have: from the position $(16,12)$  $b_{14,14}=1+b_{8,14}$; from $(10,16)$  $b_{6,16}(a_{9,8}+a_{9,7}-a_{9,6}-a_{9,5}) \Rightarrow b_{6,16}=0$; from $(16,9)$ $b_{15,5}(a_{6,10}+a_{5,10})=0\Rightarrow b_{15,5}=0$.

Now by block-diagonal basis changes, that are not identical on long roots, we can get $b_{13,1}=0$, $b_{13,2}=-1$.

After it we consider the positions $(1,14)$ and $(2,14)$ of the commutation condition, their sum is $(a_{1,1}+a_{1,2}-1)(b_{1,14}+b_{2,14})=0$. Since $b_{1,14}+b_{2,14}$ is invertible, we have $a_{1,2}=1-a_{1,1}$.

Now similarly to the previous section, we will write the conditions modulo radical, sequentially excepting variables.

From the commutation condition: the position $(13,1)$ gives $a_{1,1}=1$; $(13,3)$ gives $a_{1,3}=1$;  $(13,5)$ gives $a_{1,5}=1$; $(13,6)$ gives $a_{1,6}=1$. From  $Con1$: the position $(2,5)$ gives $a_{1,9}=0$; again the commutation condition: the position $(2,9)$ gives $b_{1,5}=0$;  $(3,9)$ gives $b_{3,5}=0$;  $(5,9)$ gives $b_{5,6}=0$; gives $(6,9)$ gives $b_{5,7}=0$; gives $(3,14)$ gives $a_{3,1}=0$; $(5,14)$ gives $a_{5,1}=0$;  $(6,14)$ gives $a_{6,1}=0$; $(9,14)$ gives $a_{9,1}=b_{8,14}$; $(13,10)$ gives $b_{13,5}=0$; $Con1$: the position  $(3,3)$ gives $a_{3,3}=1$;  $(3,5)$ gives $a_{3,5}=a_{3,9}$;  $(3,6)$ gives $a_{3,6}=0$;
 $(10,9)$ gives $a_{9,10}=-a_{9,6}$; the condition $Con2=(h_{e_2-e_3}(-1)x_1h_{e_2-e_3}(-1)x_1=E)$: the position $(13,6)$ gives $b_{2,5}=0$;  $(5,5)$ gives $b_{5,5}=1$;  $(1,2)$ gives $b_{1,2}=b_{1,14}$;  $(15,14)$ gives $b_{8,14}=0$;  $(8,14)$ gives $b_{5,1}=0$;  $(4,14)$ gives $b_{3,1}=0$;  $(2,1)$ gives $b_{2,1}=0$;  $(1,1)$ gives $b_{1,1}=1$;  $(1,2)$ gives $b_{2,2}=1$.

Now introduce a new condition. Note that $[x_{e_3}(1),x_{e_4}(1)]=x_{e_3+e_4}(1)^2$, and the elements $x_{e_3+e_4}(1)^2$ and $x_{e_1-e_2}(1)^2$ are conjugate with the element of the Weil group. Therefore the matrices $x_1^2$ and $[x_4,w_{e_3-e_4}x_4 w_{e_3-e_4}^{-1}]$ are conjugate with the element of the Weil group. In particular, this element changes roots $\{e_1,-e_1,e_2,-e_2\}$ and $\{-e_3,e_3,-e_4,e_4\}$, consequently the matrix
$[x_4,w_{e_3-e_4}x_4 w_{e_3-e_4}^{-1}]$ is block-diagonal up to decomposition on the basis part  $\{-e_3,e_3,-e_4,e_4\}$ and other parts.

Therefore directly  $a_{9,6}=0$, $a_{9,9}=1$, $a_{9,8}=a_{5,9}$, $a_{5,10}=0$, $a_{5,9}=0$; $a_{5,6}=0$.

Again from $Con1$ we have $a_{6,6}=1$, $a_{9,7}=0$, $a_{5,7}=0$, $a_{6,7}=0$, $a_{5,8}=0$, $a_{5,3}=0$, $a_{5,5}=1$, $a_{3,9}=0$, $a_{3,4}=0$.

From the commutation condition $b_{5,2}=0$, $b_{5,3}=0$.
 From $Con2$ $b_{2,14}=0$, $b_{3,4}=0$, $b_{2,3}=0$, $b_{3,3}=1$, $b_{13,14}=0$, $b_{3,14}=2b_{3,2}+b_{4,14}$, $b_{1,3}=-2b_{13,3}+2b_{15,3}$.

Again from the block-diagonal matrix $[x_4,w_{e_3-e_4}x_4 w_{e_3-e_4}^{-1}]$ we get $a_{6,5}=0$, $a_{6,9}=0$, $a_{12,12}=1$, $a_{6,3}=0$, $a_{9,3}=0$, $a_{6,8}=a_{11,12}$, $a_{11,12}=a_{6,10}a_{9,5}/2$, $b_{3,2}=0$, $b_{13,3}=0$, $b_{15,3}=0$, $b_{4,14}=0$.

From other equations $a_{6,10}=-2/a_{9,5}$, $a_{9,5}=-2$, $a_{1,14}=2$.

Therefore, $\varphi_2(x_{\alpha_4}(1))=x_{\alpha_4}(1)$ and $\varphi_2(x_{\alpha_1}(1))=x_{\alpha_1}(1)$ on the basis parts under consideration. But for $x_{\alpha_4}(1)$ other basis parts (not considered yet)    
$$
\{ v_{e_1}, v_{-e_1}, v_{e_1-e_4}, v_{e_4-e_1}, v_{e_1+e_4}, v_{-e_1-e_4}, v_{e_2-e_3}, v_{e_3-e_2}, v_{e_2+e_3}, v_{-e_2-e_3}\}
$$ 
and 
$$
\{ v_{e_2}, v_{-e_2} ,v_{e_2-e_4}, v_{e_4-e_2}, v_{e_2+e_4}, v_{-e_2-e_4}, v_{e_1-e_3}, v_{e_3-e_1}, v_{e_1+e_3}, v_{-e_1-e_3}\}
$$ 
are conjugate to the basis part 
$$
\{ v_{e_1-e_2}, v_{e_2-e_1}, v_{e_1+e_2}, v_{-e_1-e_2}, v_{e_3}, v_{-e_3}, v_{e_3-e_4}, v_{e_4-e_3}, v_{e_3+e_4}, v_{-e_3-e_4}\}
$$ 
with elements of Weil group $w_{\alpha_2}$ and $w_{\alpha_1}w_{\alpha_2}$. Thus $x_4$ coincides with $x_{\alpha_4}(1)$ on the whole  $36$-dimensional space. Clear that it follows directly $\varphi_2(x_\alpha(1))=x_\alpha(1)$ for all short roots~$\alpha\in \Phi$. Besides, $\varphi_2(x_{\alpha}(2))= x_{\alpha}(2)$ for all long roots~$\alpha$.

Consequently, if we show, that under~$\varphi_2$ the diagonal matrix $h_{e_2-e_3}(2)$ is mapped to itself, the obtained statement will be proved. Consider $d_t=\varphi_2(h_{\alpha_1}(t))$.

Since $d_t$ commutes with $h_1,h_2,h_3$, that it is decomposed into diagonal blocks, corresponding to the basis parts $\{ v_{e_i}, v_{-e_i}\}$, $i=1,\dots,4$, $\{ V_1, V_2, V_3, V_4\}$  and $\{ v_{e_i-e_j}, v_{e_j-e_i}, v_{e_i+e_j}, v_{-e_i-e_j}$,  $v_{e_k-e_l}, v_{e_l-e_k}, v_{e_k+e_l}, v_{-e_k-e_l}\}$, where $i,j,k,l$ are distinct numbers from $1$ to~$4$. Since $d_t$ commutes also with $w_{e_4}$, $w_{e_3-e_4}$, $w_{e_1+e_2}$, then it has the same form on the basis parts $\{ v_{e_1}, v_{-e_1}\}$ and $\{ v_{-e_2}, v_{e_2}\}$, $\{ v_{e_3}, v_{-e_3}\}$ и $\{ v_{e_4}, v_{-e_4}\}$,\  $\{ v_{e_1-e_3}, v_{e_3-e_1}, v_{e_1+e_3}, v_{-e_1-e_3},  v_{e_2-e_4}, v_{e_4-e_2}$, $v_{e_2+e_4}, v_{-e_2-e_4}\}$ and $\{ v_{e_1-e_4}, v_{e_4-e_1}, v_{e_1+e_4}, v_{-e_1-e_4},  v_{e_2-e_3}, v_{e_3-e_2}, v_{e_2+e_3}, v_{-e_2-e_3}\}$. Also if on the part $\{ v_{e_4}, v_{-e_4}\}$ the matrix $d_t$ has the form
$$
\begin{pmatrix}
d_1& d_2\\
d_3& d_4
\end{pmatrix},
$$
then since it commutes with $w_{e_4}$, we directly have $d_4=d_1$, $d_3=d_2$. From the condition $w_{\alpha_1}d_tw_{\alpha_1}^{-1}d_t=E$ we get $2d_1d_2=0$ and $d_1^2+d_2^2=1$, therefore $d_1=1$, $d_2=0$, on the given basis part we get the obtained form.

Besides we note that $d_t$ commutes with $x_{e_1+e_2}(2)$, $x_{e_4}(1)$, $x_{e_3-e_4}(2)$, or, what is the same, with $X_{e_1+e_2}$, $X_{e_4}$, $X_{e_3-e_4}$.

Considering the basis part $\{ v_{e_1}, v_{-e_1}, v_{e_2}, v_{-e_2}\}$, we see that from commutation with $w_{e_1+e_2}$ on it the matrix $d_t$ is
$$
\begin{pmatrix}
d_{1,1}& d_{1,2}& 0& 0\\
d_{2,1}& d_{2,2}& 0& 0\\
0& 0& d_{2,2}& d_{2,1}\\
0& 0& d_{1,2}& d_{1,1}
\end{pmatrix};
$$
commutation with $X_{e_1+e_2}$, which is $e_{v_{e_1},v_{-e_2}}-e_{v_{e_2},v_{-e_1}}$ on this basis part, gives $d_{1,2}=d_{2,1}=0$,
and, finally, the condition $w_{\alpha_1}d_tw_{\alpha_1}^{-1}d_t=E$ gives $d_{1,1}d_{2,2}=1$.

Now come to the basis part $\{ V_1, V_2, V_3, V_4\}$. From commuting with $w_{e_4}$, $w_{e_3-e_4}$, $w_{e_1+e_2}$ and from the condition $w_{\alpha_1}d_tw_{\alpha_1}^{-1}d_t=E$ we easily get that on this part $d_t$ is identical.

Now we only need to consider the $8\times 8$ basis parts. At first, consider the part 
$$
\{ v_{e_1-e_2}, v_{e_2-e_1}, v_{e_1+e_2}, v_{-e_1-e_2},  v_{e_3-e_4}, v_{e_4-e_3}, v_{e_3+e_4}, v_{-e_3-e_4}\}.
 $$
We temporarily unite it with the previous basis part and get a $12\times 12$ matrix with an identical $4\times 4$ last block.

Since $d_t$ on this basis part commutes with $X_{e_1+e_2}=2e_{v_{e_1+e_2},v_{h_2}}-e_{v_{h_2},v_{-e_1-e_2}}$ и c $X_{-e_1-e_2}=2e_{v_{-e_1-e_2},v_{h_2}}-e_{v_{h_2},v_{e_1+e_2}}$, we directly have that the matrix is decomposed into diagonal blocks up to the basis parts $\{ v_{e_1-e_2}, v_{e_2-e_1},   v_{e_3-e_4}, v_{e_4-e_3}, v_{e_3+e_4}, v_{-e_3-e_4}\}$ и $v_{e_1+e_2}, v_{-e_1-e_2}$, and on the last part it is identical.

Similarly, since $d_t$ commutes with $X_{e_3-e_4}=2e_{v_{e_3-e_4},v_{h_3}}-e_{v_{e_3-e_4},v_{h_2}}-e_{v_{e_3-e_4},v_{h_4}}-e_{v_{h_3},v_{e_4-e_3}}$ and with $X_{e_4-e_3}=2e_{v_{e_4-e_3},v_{h_3}}-e_{v_{e_4-e_3},v_{h_2}}-e_{v_{e_4-e_3},v_{h_4}}-e_{v_{h_3},v_{e_3-e_4}}$, we have that the matrix is decomposed also into diagonal blocks  up to the basis parts $\{ v_{e_1-e_2}, v_{e_2-e_1}$,   $ v_{e_3+e_4}, v_{-e_3-e_4}\}$ and $v_{e_3-e_4}, v_{e_4-e_3}$, on the last part it is identical.

Now  use the fact that our matrix commutes with $w_{e_4}$, which changes pairs of roots $e_3-e_4$ and $e_3+e_4$, $e_4-e_3$ and $-e_3-e_4$. We get that  $d_t$ is identical everywhere except the block $\{ v_{e_1-e_2}, v_{e_2-e_1}\}$. This block we do not study yet, and come to the last not considered yet basis parts
$$
\{ v_{e_1-e_3}, v_{e_3-e_1}, v_{e_1+e_3}, v_{-e_1-e_3},  v_{e_2-e_4}, v_{e_4-e_2}, v_{e_2+e_4}, v_{-e_2-e_4}\}
$$ 
and 
$$
\{ v_{e_1-e_4}, v_{e_4-e_1}, v_{e_1+e_4}, v_{-e_1-e_4},  v_{e_2-e_3}, v_{e_3-e_2}, v_{e_2+e_3}, v_{-e_2-e_3}\}.
$$

At first $d_t$ commutes with $X_{e_1+e_2}$, $X_{-e_1-e_2}$, $X_{e_3+e_4}$, $X_{-e_3-e_4}$, $X_{e_3-e_4}$, $X_{e_4-e_3}$. Directly from it we have that $d_t$ is diagonal and has the form
$$
\diag [d_1,d_2,d_1,d_2,d_2,d_1, d_2,d_1,d_1,d_2, d_1,d_2,d_2,d_1,d_2,d_1].
$$
Since $w_{e_1-e_2}d_tw_{e_1-e_2}^{-1}d_t=E$, we have $d_1d_2=1$.

Now temporarily consider only the basis part $\{ v_{e_1-e_2}, v_{e_2-e_1}, v_{e_2-e_3}, v_{e_3-e_2}, v_{e_1-e_3}, v_{e_3-e_1}\}$. We have the condition $w_{e_2-e_3}d_tw_{e_2-e_3}^{-1}=d_tw_{e_1-e_2}w_{e_2-e_3}d_tw_{e_2-e_3}^{-1}w_{e_1-e_2}^{-1}$. It implies that on the block $\{ v_{e_1-e_2}, v_{e_2-e_1}\}$ the matrix $d_t$ is diagonal with numbers $d_1^2, 1/d_1^2$ on the diagonal.

Then use the commutation with $X_{e_4}$ and get $d_{1,1}=d_1$. Therefore $d_t=h_{\alpha_1}(d_1)$.

Note that if $t=2$, then we also have
$$
d_2 x_{e_2}(1) d_2^{-1}=x_{e_2}(1)^2,
$$
so $d_2=h_{\alpha_1}(2)$. It shows that $x_1=x_{\alpha_1}(1)$.

The obtained statement is proved.

\subsection{The case $B_l$, $l\geqslant 5$.}

In this case we consider $x=\varphi_2(x_{e_1-e_2}(1))$ and $y=\varphi_2(x_{e_3}(1))$.

The matrix $x$ commutes with $h_{\alpha_1}(-1)$, $h_{\alpha_3}(-1), h_{\alpha_4}(-1),\dots, h_{\alpha_{l-1}}(-1)$, therefore it is decomposed into diagonal blocks corresponding to the following basis parts:

1. $\{ v_{e_1-e_2}, v_{e_2-e_1}, v_{e_1+e_2}, v_{-e_1-e_2}, V_{h_1},\dots, V_{h_l}\}$;

2. $\{ v_{e_i-e_j}, v_{e_j-e_i}, v_{e_i+e_j}, v_{-e_i-e_j}\}$, $2< i< j\leqslant l$;

3. $\{ v_{e_1-e_i}, v_{e_i-e_1}, v_{e_1+e_i}, v_{-e_1-e_i}, v_{e_2-e_i}, v_{e_i-e_2}, v_{e_2+e_i}, v_{-e_2-e_i}\}$, $2< i\leqslant l$;

4. $\{ v_{e_1}, v_{-e_1}, v_{e_2}, v_{-e_2}\}$;

5. $\{ v_{e_i}, v_{-e_i}\}$, $2< i\leqslant l$.

If on the basis part of the last type the matrix $x$ has the form
$$
\begin{pmatrix}
a& b\\
c& d
\end{pmatrix},
$$
then from the condition $h_{\alpha_2}(-1)xh_{\alpha_2}(-1)x=E$ we have $c(a+d)=b(a+d)=0$, consequently (since $a+d\equiv 2\mod J$) we get $b=c=0$.
Then $a^2=d^2=1$, therefore $a=d=1$. Thus, on all basis parts of the fifth type $x$ is identical.

Now consider the basis part of the fourth type.

According to commutation with $w_{e_1+e_2}$ we directly have that on this basis part $x$ is
$$
\begin{pmatrix}
x_1& x_2& x_3& x_4\\
x_5& x_6& x_7& x_8\\
-x_8& -x_7& x_6& x_5\\
-x_4& -x_3& x_2& x_1
\end{pmatrix}.
$$

Let us unite this basis part with $\{ v_{e_3}, v_{-e_3}\}$, where $x$ is identical, as we already know.
On this basis part we consider the condition
$$
xw_{e_1-e_2}w_{e_2-e_3}xw_{e_2-e_3}^{-1}w_{e_1-e_2}^{-1}=w_{e_2-e_3}xw_{e_2-e_3}^{-1}w_{e_1-e_2}w_{e_2-e_3}xw_{e_2-e_3}^{-1}w_{e_1-e_2}^{-1}x.
$$
From its position $(6,5)$ it follows $x_2(1-x_1-x_6)=0$. Since $1-x_1-x_6\equiv -1\mod J$, then $x_2=0$. Similarly from the position $(5,6)$ we have $x_5=0$. After that from the positions $(5,5)$ and $(6,6)$ $x_1=x_6=1$. From $(1,4)$ $x_3(x_4+x_7)=0$, and since $x_3\equiv -1 \mod J$, then $x_7=-x_4$. From $(6,2)$ and $(5,1)$ we get $x_4^2=x_3+x_3^2$ and $x_4^2=x_8+x_8^2$, therefore $(x_3-x_8)(1+x_3+x_8)=0$.
Since $x_3-x_8\equiv -1\mod J$, then $x_8=-1-x_3$.

Therefore now we have that on the basis part of the fourth type $x$ is
$$
\begin{pmatrix}
1& 0& x_3& x_4\\
0& 1& -x_4& -1-x_3\\
1+x_3& x_4& 1& 0\\
-x_4& -x_3& 0& 1
\end{pmatrix}.
$$

Since elements of the Weil group which do not move~$\alpha_1$, act transitively on the roots of the same length, orthogonal to~$\alpha_1$, we have that on all basis parts of the second type and on all basis parts of the third type $x$ has the same form.

Consider any basis part of the second type (for example, $\{ v_{e_3-e_4}, v_{e_4-e_3}$,\\ $v_{e_3+e_4}, v_{-e_3-e_4}\}$). Since $x$ commutes with $w_{e_3}$, $w_{e_4}$, $w_{e_3-e_4}$ (in the general case with $w_{e_i}$, $w_{e_j}$, $w_{e_i-e_j}$), it follows that $x$ has the form
$$
\begin{pmatrix}
x_5& x_6& x_7& -x_7\\
x_6& x_5& -x_7& x_7\\
x_7& -x_7& x_5& x_6\\
-x_7& x_7& x_6& x_5
\end{pmatrix}.
$$

Using the condition $h_{e_2-e_3}(-1)x h_{e_2-e_3}(-1)x=E$ we get (on the place $(1,3)$) $2x_7(x_5-x_6)=0$, therefore, since $x_5-x_6\equiv 1\mod J$, we have $x_7=0$. After that from the position $(1,2)$ it follows $2x_5x_6=0$, so $x_6=0$, and then directly $x_5=0$. Consequently we also know that on the basis parts of the second type $x$ is identical.

Let us come to the basis part of the third type (similarly, we can take $i=3$).

According to commutation with $w_{e_1+e_2}$ and $w_{e_i}$ we have that $x$ on this basis part has the form
$$
\begin{pmatrix}
x_5& x_6& x_7& x_8& x_9& x_{10}& x_{11}& x_{12}\\
x_{13}& x_{14}& x_{15}& x_{16}& x_{17}& x_{18}& x_{19}& x_{20}\\
x_7& x_8& x_5& x_6& x_{11}& x_{12}& x_9& x_{10}\\
 x_{15}& x_{16}& x_{13}& x_{14}&  x_{19}& x_{20}&  x_{17}& x_{18}\\
-x_{18}& -x_{17}& -x_{20}& -x_{19}& x_{14}& x_{13}& x_{16}& x_{15}\\
-x_{10}& -x_9& -x_{12}& -x_{11}& x_6& x_5& x_8& x_7\\
-x_{20}& -x_{19}& -x_{18}& -x_{17}& x_{16}& x_{15}& x_{14}& x_{13}\\
-x_{12}& -x_{11}& -x_{10}& -x_9& x_8& x_7& x_6& x_5
\end{pmatrix}.
$$

Now consider the last basis part  $\{ v_{e_1-e_2}, v_{e_2-e_1}, v_{e_1+e_2}, v_{-e_1-e_2}, V_{h_1},\dots, V_{h_l}\}$. According  to commutation with $w_{\alpha_3}$, \dots, $w_{\alpha_l}$ the matrix $x$ is identical on the basis part $\{V_{h_3},\dots, V_{h_l}\}$. Therefore, we can consider in the given case the $6\times 6$ matrix.

Commuting with $w_{e_1+e_2}$ gives us the following form:
$$
\begin{pmatrix}
x_{21}& x_{22}& x_{23}& -x_{23}& -2x_{24}& x_{24}\\
x_{25}& x_{26}& x_{27}& -x_{27}&  -2x_{28}& x_{28}\\
x_{29}& x_{30}& x_{31}& x_{32}& x_{33}& x_{34}\\
-x_{29}& -x_{30}& x_{32}&  x_{31}& -x_{33}& x_{33}+x_{34}\\
x_{35}& x_{36}& x_{37}& x_{38}& x_{39}& x_{40}\\
0& 0& x_{37}+x_{38}& x_{37}+x_{38}& 0& 2x_{40}+x_{39}
\end{pmatrix}.
$$

Therefore for the matrix $x$ we now have $40$ variables $x_1,\dots, x_{40}$.

Now consider the matrix $y$. It commutes with $h_{\alpha_1}(-1)$, $h_{\alpha_2}(-1)h_{\alpha_3}(-1)$, $h_{\alpha_4}(-1)$, $h_{\alpha_5}(-1)$,  \dots, $h_{\alpha_{l-1}}(-1)$ and is decomposed to the following blocks:

1. $\{ v_{e_3}, v_{-e_3}, V_{h_1},\dots, V_{h_l}\}$;

2. $\{v_{e_3-e_i}, v_{e_i-e_3}, v_{e_3+e_i}, v_{-e_3-e_i},  v_{e_i}, v_{-e_i}\}$, $i\ne 3$;

3. $\{ v_{e_i-e_j}, v_{e_j-e_i}, v_{e_i+e_j}, v_{-e_i-e_j}\}$, $1\leqslant i< j\leqslant l$; $i,j\ne 3$.

On the basis parts of the third type $y$ is identical (it is proved just by the same way as for the matrix~$x$; we use here commutation with $w_{e_i}$, $w_{e_j}$, $w_{e_i-e_j}$ and the condition $h_{e_2-e_3}(-1) y h_{e_2-e_3}(-1)y=E$).

Clear that on all basis parts of the second type the matrix  $y$ have the same form (since the elements of the Weil, which do not move~$e_3$, act transitively on the roots $e_i$, $i\ne 3$).

According to commutation with $w_{e_3}$ the matrix $y$ on the given basis part has the form
$$
\begin{pmatrix}
y_1& y_2& y_3& y_4& y_5& y_6\\
y_7& y_8& y_9& y_{10}& y_{11}& y_{12}\\
y_3& y_4& y_1& y_2& -y_6& -y_5\\
y_9& y_{10}& y_7& y_8& -y_{12}& -y_{11}\\
y_{13}& y_{14}& y_{15}& y_{16}& y_{17}& y_{18}\\
-y_{15}& -y_{16}& -y_{13}& -y_{14}& y_{18}& y_{17}
\end{pmatrix}.
$$

On the basis part of the first type, according to commutation with all $w_{\alpha_j}$, which do not move~$e_3$, on the basis elements $V_{h_1}, V_{h_5}, \dots, V_{h_l}$ the matrix $y$ is identical. On the basis part $\{ v_{e_3}, v_{-e_3}, V_{h_2},V_{h_3}, V_{h_4}\}$ we use additionally  $h_{e_2-e_3}(-1) y h_{e_2-e_3}(-1)y=E$, $yw_{e_4}yw_{e_4}^{-1}y=w_{e_4}$, and also, as in the previous sections, make basis changes so that on the places $(V_{h_3},v_{e_3})$ and $(V_{h_4},v_{e_3})$ there are zeros, and on the places $(V_{h_3},v_{-e_3})$, $(V_{h_4},v_{-e_3})$ there are $-1$. Then we have the following $y$:
$$
\begin{pmatrix}
1& y_{19}& y_{19}-y_{19}^2& y_{19}^2-y_{19}& 0\\
0& -y_{19}& y_{19}^2-1& 1-y_{19}^2& 0\\
0& 0& 1& 0& 0\\
0& -1& 1+y_{19}& -y_{19} & 0\\
0& -1& 1+y_{19}& -1-y_{19}& 1
\end{pmatrix}.
$$

Consequently for the matrix $y$ we have $19$ variables $y_1,\dots, y_{19}$.

With one more basis change from the previous section let us make $x_{35}=0$, $x_{36}=-1$.

Now use the fact that $x$ and $y$ commute. At first, consider the basis part 
$$
\{ v_{e_1-e_2}, v_{e_2-e_1}, v_{e_1+e_2}, v_{-e_1-e_2}, v_{e_3}, v_{-e_3}, V_{h_1}, V_{h_2}, V_{h_3}, V_{h_4}\}.
 $$
The matrix $y$ on it has only one unknown variable $y_{19}\equiv -1\mod J$. According to the given condition  $x_{37}+x_{38}=0$ and $x_{39}+2x_{40}=1$.

Consider now the basis part  $v_{e_1-e_2}, v_{e_2-e_1}, v_{e_1+e_2}, v_{-e_1-e_2}, v_{e_1}, v_{-e_1}, v_{e_2}, v_{-e_2}, V_{h_1}, V_{h_2}$. Since the matrix $x_{e_1}(1)$ is conjugate to $x_{e_3}(1)$  with the help of the Weil group element, we easily can take $\varphi_2(x_{e_1}(1))=y'$. The matrix  $y'$ has the same variables as~$y$. The matrices $x$ and $y'$ also commute. Consider this condition on the basis part under consideration.

The matrix $x$ is
$$
\begin{pmatrix}
x_{21}& x_{22}& x_{23}& -x_{23}& 0& 0& 0& 0& -2x_{24}& x_{24}\\
x_{25}& x_{26}& x_{27}& -x_{27}& 0& 0& 0& 0& -2x_{28}& x_{28}\\
x_{29} x_{30}& x_{31}& x_{32}& 0& 0& 0& 0& x_{33}& x_{34}\\
-x_{29}& -x_{30}& x_{32}& x_{31}& 0& 0& 0& 0& -x_{33}& x_{33}+x_{34}\\
0& 0& 0& 0& 1& 0& x_3& x_4& 0& 0\\
0& 0& 0& 0& 0& 1& -x_4& 1-x_3& 0& 0\\
0& 0& 0& 0& 1+x_3& x_4& 1& 0& 0& 0\\
0& 0& 0& 0& -x_4& -x_3& 0& 1& 0& 0\\
0& -1& x_{37} -x_{37}& 0& 0& 0& 0& 1-2x_{40}& x_{40}\\
0& 0& 0& 0& 0& 0& 0& 0& 0& 1
\end{pmatrix},
$$
and $y'$ is
$$
\begin{pmatrix}
y_1& y_2& y_3& y_4& 0& 0& y_5& y_6& 0& 0\\
y_7& y_8& y_9& y_{10}& 0& 0& y_{11}& y_{12}& 0& 0\\
y_3& y_4& y_1& y_2& 0& 0& -y_6& -y_5& 0& 0\\
y_9& y_{10}& y_7& y_8& 0& 0& -y_{12}& -y_{11}& 0& 0\\
0& 0& 0& 0& 1& y_{19}& 0& 0& y_{19}^2-y_{19}& 0\\
0& 0& 0& 0& 0& -y_{19}& 0& 0& y_{19}^2-1& 0\\
y_{13}& y_{14}& y_{15}& y_{16}& 0& 0& y_{17}& y_{18}& 0& 0\\
-y_{15}& -y_{16}& -y_{13}& -y_{14}& 0& 0& y_{18}& y_{17}& 0& 0\\
0& 0& 0& 0& 0& -1& 0& 1+y_{19}& -y_{19}& 0\\
0& 0& 0& 0& 0& 0& 0& 0& 0& 1
\end{pmatrix}.
$$

The position $(5,10)$ of commutation gives $x_{40}(y_{19}-y_{19}^2)=0$, therefore $x_{40}=0$. The position $(9,6)$ gives $x_3(1+y_{19})=0$, so $y_{19}=-1$.
If we consider the condition $h_{e_2-e_3}(-1)xh_{e_2-e_3}(-1)x=E$, we directly get $x_4=0$, $x_3=1$.
Now from the commutation condition (the fifth line) $y_{13}=0$, $y_{14}=-2$, $y_{15}=2x_{37}$, $y_{16}=-2x_{37}$, $y_{17}=1$, $y_{18}=0$, (the sixth row) $y_6=-2x_{24}$, $y_{12}=-2x_{28}$, $y_5=-x_{33}$, $y_{11}=x_{33}$. From the seventh line of the condition $h_{e_2-e_3}(-1)xh_{e_2-e_3}(-1)x=E$ we have: $x_{25}=2x_{37}x_{29}$, $x_{26}=1+2x_{37}x_{30}$, $x_{27}=x_{37}(-1+x_{31}-x_{32})$, $x_{28}=-x_{33}x_{37}$, $x_{29}(-x_{21}-2x_{30}x_{37}+x_{32}-x_{31})=0$, so $x_{29}=0$, then $x_{21}=1$.

From the position $(2,10)$ of
$$
Comm: (y' w_{e_1-e_2} y' w_{e_1-e_2}^{-1}= w_{e_2} x^2 w_{e_2}^{-1} w_{e_1-e_2} y' w_{e_1-e_2}^{-1} y')
$$
we have $x_{34}=x_{33}$, and from the position $(2,9)$ $x_{33}(-1-4x_{37}x_{30}-2x_{31}+2x_{32})=0$, therefore $x_{33}=0$.

Again from the commutation condition $y_9x_{24}=x_{32}x_{24}=x_{37}x_{24}=0$, so $y_9=x_{32}=x_{37}=0$, then $x_{31}=1$, $y_7=y_{10}=y_3=x_{23}=x_{30}=0$, $y_1=y_8=1$. From $h_{e_2-e_3}(-1)xh_{e_2-e_3}(-1)x=E$ we have $x_{24}=x_{22}$, and from $h_{e_1-e_2}(-1)y'h_{e_1-e_2}(-1)y'=E$ we have $y_2=0$, $y_4=-2x_{22}$. Finally, from $Comm$ we see that $x_{22}=-1$, and we have now variables.

Now the matrix  $y'$ is equal to $x_{e_1}(1)$ on the basis part under consideration, consequently, on the whole basis. Respectively, the matrix $y$ is $x_{e_3}(1)$.

Now we need only to find~$x$ on the basis parts of the third part. It is easily obtained from the commutation condition and the form of~$x^2$, which is already known.

Similarly to the previous sections we can see that $\varphi_2(h_{\alpha}(t))=h_{\alpha}(s)$ for every $t\in R^*$ and $\alpha\in \Phi$.

\section{Images of $x_{\alpha_i}(t)$,
proof of Theorem 2.}

We have shown that $\varphi_2(h_{\alpha}(t))=h_{\alpha}(s)$,
$\alpha\in \Phi$, $t\in R^*$. Denote the mapping $t\mapsto s$ by $\rho: R^*
\to R^*$. Note that for $t\in R^*$
$\varphi_2(x_{e_1}(t))=\varphi_2(h_{e_1-e_2}(t^{-1}) x_{e_1}(1)
h_{e_1-e_2}(t))=h_{e_1-e_2}(s^{-1}) x_{e_1}(1) h_{e_1-e_2}(s)=
x_{e_1}(s)$. If $t\notin R^*$, then $t\in J$, i.\,e., $t=1+t_1$, where
$t_1\in R^*$. Then
$\varphi_2(x_{e_1}(t))=\varphi_2(x_{e_1}(1)x_{e_1}(t_1))=x_{e_1}(1)x_{e_1}(\rho(t_1))=
x_{e_1}(1+\rho(t_1))$. Therefore if we extend the mapping
$\rho$ to the whole~$R$ (by the formula $\rho(t):=1+\rho(t-1)$, $t\in
R$), we obtain $\varphi_2(x_{e_1}(t))=x_{e_1}(\rho(t))$ for all $t\in R$. Then (since all short roots are conjugate under the action of the Weil group) $\varphi_2(x_{\alpha}(t))=x_{\alpha}(\rho(t))$ for all short roots~$\alpha$. According to the commutation condition
$$
x_{e_i\pm e_j}(2t)=[x_{e_i}(t), x_{\pm e_j}(1)]
$$
we get  $\varphi_2(x_{\alpha}(t))=x_{\alpha}(\rho(t))$ for all long roots~$\alpha$.

Clear that $\rho$ is injective, additive, and also  multiplicative on all invertible elements. Since every element of~$R$ is a sum of two invertible elements, we have that  $\rho$ is an isomorphism from the ring~$R$ onto some its subring~$R'$. Note that in this situation $C E(\Phi,R) C^{-1}=E(\Phi,R')$ for some matrix $C\in
\GL(V)$. Let us show that $R'=R$.

Denote matrix units by $E_{ij}$.

\begin{lemma}\label{porozhd}
The elementary Chevalley group $E_{\ad}(\Phi,R)$ generates the matrix ring $M_n(R)$.
\end{lemma}
\begin{proof}
The martix $(x_{\alpha_1}(1)-1)^2$ has a unique nonzero element
$-2\cdot E_{12}$. Multiplying it to suitable diagonal matrices, we can obtain
an arbitrary matrix of the form $\lambda\cdot E_{12}$ (since $-2\in R^*$ and $R^*$ generates~$R$). Since the Weil group acts transitively on all roots of the same length, i.\.e., for every long root $\alpha_k$ there exists such  $w\in W$, that
$w(\alpha_1)=\alpha_k$, and then the matrix $\lambda E_{12}\cdot w$ has the form
$\lambda E_{1,2k}$,  and the matrix $w^{-1}\cdot \lambda E_{12}$ has the form
$\lambda E_{2k-1,2}$. Besides, with the help of the Weil group element, moving the first root to the opposite one, we can get the matrix unit
$E_{2,1}$. Taking now different combinations of the obtained elements, we can get an arbitrary element $\lambda E_{ij}$, $1\leqslant i,j\leqslant 2m$, indices $i,j$ correspond to the numbers of long roots.

The matrix $(x_{e_i}(1)-1)^2$  is
$$
-2E_{v_{e_i},v_{-e_i}}+\sum_{j\ne i} (2E_{v_{e_i-e_j},v_{-e_i-e_j}}+2E_{v_{e_i+e_j},v_{e_j-e_i}}).
$$
All matrix units in this sum, except the first one, are already obtained, therefore we can subtract them and get $E_{v_{e_i},v_{-e_i}}$.
Similarly to the longs roots, using the fact that all short roots are also conjugate under the action of the Weil groups, we obtain all $\lambda E_{ij}$, $1\leqslant i,j\leqslant 2m$, indices $i,j$ correspond to the short roots.

Now subtract from the matrix $x_{\alpha_1}(1)-1$ suitable matrix units and obtain the matrix $E_{V_{h_1}, v_{-\alpha_1}}-2 E_{v_{\alpha_1},V_{h_1}}+ E_{v_{\alpha_1},V_{h_2}}$. Multiplying it (from the right side) to $E_{v_{-\alpha_1},i}$, $1\leqslant i\leqslant 2m$, where $i$ corresponds to a long root, we obtain all
$E_{V_{h_1}, i}$, $1\leqslant i\leqslant 2m$ for $i$ corresponding to the long roots. Multiplying these last elements from the left side to $w_{\alpha_2}$, we obtain $E_{V_{h_2}, i}$, $1\leqslant i\leqslant 2m$ for $i$, corresponding to the long roots; then, similarly multiplying the obtained elements from the left side to $w_{\alpha_3}, \dots, w_{\alpha_l}$, we get all  $E_{V_{h_j}, i}$, $1\leqslant j \leqslant l$, $1\leqslant i\leqslant 2m$ for  $i$, corresponding to the long roots.

Then
\begin{align*}
A&=1/2^l(h_{\alpha_1}(-1)+E)\dots (h_{\alpha_l}(-1)+E)=E_{V_{h_1},V_{h_1}}+\dots+E_{V_{h_l},V_{h_l}},\\ B&=A(w_{\alpha_1}-A)(w_{\alpha_2}-A)\dots (w_{\alpha_{l-1}}-A)(w_{\alpha_l}-A)(w_{\alpha_{l-1}}-A) \dots\\
 &\quad\quad  \dots(w_{\alpha_2}-A)(w_{\alpha_1}-A)A=-4E_{V_{h_1},V_{h_1}}+2E_{V_{h_1},V_{h_2}},\\
C_l&=A(w_{\alpha_1}-A)(w_{\alpha_2}-A)\dots (w_{\alpha_{l-1}}-A)(w_{\alpha_l}-A)A=2E_{V_{h_1},V_{h_{l-1}}}-2E_{V_{h_1},V_{h_l}},\\
C_{l-1}&=w_{\alpha_{l-1}}C_l w_{\alpha_{l-1}}=2E_{V_{h_1},V_{h_{l-2}}}-2E_{V_{h_1},V_{h_{l-1}}},\\
\dots &\dots \dots \dots \dots\\
C_{3}&=w_{\alpha_{3}}C_4 w_{\alpha_{3}}=2E_{V_{h_1},V_{h_{2}}}-2E_{V_{h_1},V_{h_{3}}},\\
C_2&=w_{\alpha_{2}}C_3 w_{\alpha_{2}}=2E_{V_{h_1},V_{h_{1}}}+2E_{V_{h_2},V_{h_{1}}}-2E_{V_{h_1},V_{h_{2}}}-2E_{V_{h_2},V_{h_{2}}} .
 \end{align*}

 Now
\begin{align*}
C&=B+C_2=-2E_{V_{h_1},V_{h_{1}}}+2E_{V_{h_2},V_{h_{1}}}-2E_{V_{h_2},V_{h_{2}}},\\
C_1&=w_{\alpha_1}Cw_{\alpha_1}=-4E_{V_{h_1},V_{h_{1}}}+2E_{V_{h_1},V_{h_{2}}}-2E_{V_{h_2},V_{h_{1}}},
\end{align*}
finally,
$$
C_1+B=-2E_{V_{h_2},V_{h_1}}.
 $$
Now with the help of Weil group elements and multiplication of matrix units we get all  $E_{i,j}$, $2m<i,j\leqslant n$, and then all
 $E_{i,j}$, $1\leqslant i\leqslant n$, $2m< j\leqslant n$, where $i$ correspond to the long roots.

 Now taking the matrix $x_{e_1}(t)$ and multiplying it from the both sides to the suitable matrix units  $E_{i,i}$,  we can obtain $E_{i,j}$, where  $i$ corresponds to a long root,  $j$ corresponds to a short one. After that it becomes clear, how to get all matrix units $E_{i,j}$, $1\leqslant i,j \leqslant 2m$ with the help of Weil group elements. Finally, as above, we can get all   $E_{i,j}$, $1\leqslant i\leqslant 2m$, $2m< j\leqslant n$, where $i$ corresponds to the short roots, and therefore all matrix units.
\end{proof}

\begin{lemma}\label{Tema}
If for some $C\in \GL(V)$ we have $C E(\Phi,R) C^{-1}=
E(\Phi,R')$, where $R'$ is a subring in~$R$, then $R'=R$.
\end{lemma}
\begin{proof}
Suppose that $R'$ is a proper subring of~$R$.

Then $C M_n(R) C^{-1} =M_n (R')$, since the group $E(\Phi,R)$
generates the ring $M_n(R)$, and the group $E(\Phi,R')=CE(\Phi,R) C^{-1}$
generates the ring $M_n(R')$. It is impossible, since $C\in \GL_n(R)$.
\end{proof}

Therefore we have proved that  $\rho$ is an automorphism of the ring~$R$. Consequently, the composition of the initial automorphism~$\varphi$ and some basis change with a matrix $C\in \GL_n(R)$, (mapping  $E(\Phi,R)$ into itself)
is a ring automorphism~$\rho$. It proves Theorem~2.
 $\square$

\section{Proof of Theorem 3.}

In this section we still consider an elementary adjoint Chevalley group
$E_{\ad}(\Phi,R)$ of the type $B_l$, $l\geqslant 2$,  over a local ring with $1/2$.

Recall, that actually we prove that the normalizer of the Chevalley group $E_{\ad}(\Phi,R)$ in $\GL_{n}(R)$ is $\lambda\cdot G_{\ad}(\Phi,R)$.

\begin{proof}
Suppose that we have some matrix $C=(c_{i,j})\in \GL_n(R)$ such that
$$
C\cdot E_{\ad}(\Phi,R) \cdot C^{-1}=E_{\ad}(\Phi,R).
$$

If $J$ is the maximal ideal (radical) of $R$, then the matrices from
$M_n(J)$ form the radical of the matrix ring $M_n(R)$, therefore
$$
C\cdot M_n(J)\cdot C^{-1}=M_n(J),
$$
consequently
$$
C\cdot (E+M_n(J))\cdot C^{-1}=E+M_n(J),
$$
so
$$
C\cdot E_{\ad}(\Phi,R,J)\cdot C^{-1}=E_{\ad}(\Phi,R,J),
$$
since $E_{\ad}(\Phi,R,J)=E_{\ad}(\Phi,R)\cap (E+M_n(J)).$

Therefore the image $\overline C$ of~$C$ under factorization of the
ring $R$ by its radical~$J$ gives us an   automorphism--conjugation of
the Chevalley group $E_{\ad} (\Phi,k)$, where $k=R/J$ is a residue
field of~$R$.

\begin{lemma}\label{fromAnton}
If $E_{\ad}(\Phi,k)$ is a Chevalley group of type $B_l$, $l\geqslant 2$,  over a field $k$ of characteristic $\ne 2$, then every automorphism--conjugation is inner.
\end{lemma}
\begin{proof}
By Theorem~30 from~\cite{Steinberg} every automorphism of a Chevalley group of the type~$B_l$ over a field~$k$ of characteristic $\ne 2$ is standard, i.\,e., for the given type it is a composition of inner and ring automorphisms. Suppose that a matrix~$C$ is from the normalizer of $E_{\ad}(\Phi,k)$ in $\GL_{n}(k)$. Then $i_C$ is an automorphism of $E_{\ad}(\Phi,k)$, therefore we have $i_C=i_g\circ \rho$, $g\in G_{\ad}(\Phi,k)$, where $\rho$ is a ring automorphism. Consequently, $i_{g^{-1}}i_C=i_{C'}=\rho$ and some matrix $C'\in \GL_{n}(k)$ defines a ring automorphism~$\rho$. For every root $\alpha\in \Phi$ we have $\rho(x_\alpha(1))=x_\alpha(1)$, so $C'x_\alpha(1)=x_\alpha(1) C'$ for all $\alpha\in \Phi$. Thus, the matrix $C'$ is scalar, and the automorphism $i_C$ is inner.
\end{proof}

By Lemma~\ref{fromAnton}
$$
i_{\overline C}=i_g,\quad  g\in G_{\ad}(\Phi,k).
$$

Since over a field every element of the Chevalley group is a product of some element from elementary subgroup (that is generated by unipotents $x_\alpha(t)$)
and some torus element, then the matrix $g$ can be decomposed to the product
$t_{\alpha_1}(X_1)\dots t_{\alpha_l}(X_l) x_{\alpha_{i_1}}(Y_1)\dots
x_{i_N}(Y_N)$, where $X_1,\dots,X_l,Y_1,\dots, Y_N\in k$,
$t_{\alpha_k}(X)$ is a torus element, corresponding to a
homomorphism~$\chi$ such that $\chi(\alpha_k)=X$, $\chi(\alpha_j)=1$
for $j\ne k$, $1\leqslant j\leqslant l$.

Since every element $X_1,\dots, X_l, Y_1,\dots, Y_N$ is a residue
class in the ring~$R$, we can choose (by arbitrary way) elements
$x_1\in X_1$, \dots, $x_l\in X_l$, $y_1\in Y_1$, \dots, $y_N\in
Y_n$, and the element
$$
g'=t_{\alpha_1}(x_1)\dots t_{\alpha_l}(x_l)
x_{\alpha_{i_1}}(y_1)\dots x_{i_N}(y_N)
$$
satisfies the conditions $g'\in G_{\ad}(\Phi,R)$ and $\overline
{g'}=g$.

Consider the matrix $C'={g'}^{-1}\circ {A_\delta}^{-1}\circ C$. This matrix
also normalizes an elementary Chevalley group $E_{\ad}(\Phi,R)$,
moreover $\overline {C'}=E$. Therefore we reduce the description of
matrices from the normalizer of $E_{\ad}(\Phi,R)$ to the description
of matrices from this normalizer, equivalent to~$E$ modulo~$J$.

Consequently we can suppose that our initial matrix $C$ is
equivalent to~$E$ modulo~$J$.

We want to show that $C\in G_{\ad}(\Phi,R)$.

At first we want to prove one technical lemma.

\begin{lemma}\label{prod2}
Let $X=\lambda t_{\alpha_1}(s_1)\dots t_{\alpha_l}(s_l)x_{\alpha_1}(t_1)\dots x_{\alpha_m}(t_m)x_{-\alpha_1}(u_1)\dots x_{-\alpha_m}(u_m)\in \lambda E_{\ad}(\Phi,R,J)$.
Then the matrix $X$ has such $n+1$ coefficients \emph{(}precisely described in the proof of this lemma\emph{)} that uniquely define elements $\lambda$, $s_1,s_2$, $t_1,\dots,t_m$, $u_1,\dots, u_m$.
\end{lemma}

\begin{proof}
Firstly we mark a special root sequence in the system $\Phi$ of the type~$B_l$.

As we remember, every root of the system $B_l$ has the form either $\pm e_i\pm e_j$, $i\ne j$, or $\pm e_i$.

Consider the sequence
\begin{align*}
\gamma_1&=e_1+e_2=\alpha_1+2\alpha_2+\dots+2\alpha_l,\\
\gamma_2&=e_1+e_3=\alpha_1+\alpha_2+2\alpha_3+\dots+2\alpha_l,\\
&\dots \dots \dots \dots\\
\gamma_{l-1}&=e_1+e_l=\alpha_1+\alpha_2+\dots+\alpha_{l-1}+2\alpha_l,\\
\gamma_l&=e_2+e_l=\alpha_2+\dots+\alpha_{l-1}+2\alpha_l,\\
\gamma_{l+1}&=e_2=\alpha_2+\dots+\alpha_{l-1}+\alpha_l,\\
\gamma_{l+2}&=e_2-e_l=\alpha_2+\dots+\alpha_{l-1},\\
&\dots \dots \dots \dots\\
\gamma_{2l-2}&=e_2-e_4=\alpha_2+\alpha_3,\\
\gamma_{2l-1}&=e_2-e_3=\alpha_2.
\end{align*}

This sequence satisfies the following properties:

1. $\gamma_1$ is the maximal root of the system.

2. $\gamma_{2l-1}$ is a simple root.

3. Every root is obtained from the previous one by subtracting some simple root, and  the first $l$ subtracted simple roots are different.

4. All roots in the systems $A_l,D_l$ are either members of the corresponding sequences, or differences of some two roots from these sequences.

To find an element on this position in the matrix we need to write all sequences of roots $\beta_1, \dots, \beta_p$ with two following properties:

1. $\mu + \beta_1\in \Phi$, $\mu +\beta_1+\beta_2\in \Phi$, \dots, $\mu+\beta_1+\dots+\beta_i\in \Phi$, \dots,
$\mu+\beta_1+\dots+\beta_p=\nu$.

2. In our initial numerated sequence  $\alpha_1,\dots,\alpha_m, -\alpha_1,\dots, -\alpha_m$ roots $\beta_1,\dots,\beta_p$ are staying strictly from right to left.

Finally, in the matrix $X$ on the position $(\mu,\nu)$ there is sum of all products $\pm \beta_1\cdot\beta_2\dots \beta_p$ by all sequences of roots with these two properties, multiplied by $d_\mu=\lambda s_1^{\langle \alpha_1,\mu\rangle}\dots s_l^{\langle \alpha_l,\mu\rangle}$. If $\mu=\nu$, then we must add~$1$ to the sum.

We shall find the obtained numbers $\lambda, s_1,\dots,s_l,t_1,\dots, t_m,u_1,\dots,u_m$ step by step.

At first we consider in the matrix $X$ the position $(-\gamma_1,-\gamma_1)$. We can not add to the root $-\gamma_1$ any negative root so that in the results we get a root again. If in the sequence $\beta_1,\dots, \beta_p$ the first root is positive, then all other roots have to be positive. Therefore, on the place $(-\gamma_1,-\gamma_1)$ there is just $d_{-\gamma_1}$, and we can know it now. Now consider the place $(-\gamma_1,-\gamma_2)$.  By the same reason  the obtained sequence is only  $\alpha_{k_1}=\gamma_1-\gamma_2$, i.\,e., some simple root $\alpha_{k_1}$. Thus, it is $\pm d_{-\gamma_1} t_{k_1}$ on this place. So we also have found $t_{k_1}$.
If we consider the positions  $(-\gamma_2,-\gamma_2)$ and $(-\gamma_2,-\gamma_1)$, we can see that by similar reasons there are $d_{-\gamma_2}(1\pm u_{k_1}t_{k_1}$ and $\pm d_{-\gamma_2} u_{k_1}$, respectively. So we know $d_{-\gamma_2}$ and $u_{k_1}$.

Now we come to the second step.
From above arguments in the matrix $X$ on the place $(-\gamma_2,-\gamma_3)$ there is $d_{-\gamma_2}(\pm t_{k_2}\pm u_{k_1}t_{k_{1,2}})$, where $\alpha_{k_2}=\gamma_2-\gamma_3$,
$\alpha_{k_{1,2}}=\alpha_{k_1}+\alpha_{k_2}$; on the place $(-\gamma_3,-\gamma_2)$ there is $d_{-\gamma_3}(\pm u_{k_2}\pm u_{k_{1,2}}t_{k_{1}})$; on the place $(-\gamma_1,-\gamma_3)$ there is $d_{-\gamma_1}(\pm t_{k_{1,2}}\pm t_{k_{1}}t_{k_{2}})$  (the second summand can absent, if $k_1< k_2$); on the place $(-\gamma_3,-\gamma_1)$ there is $d_{-\gamma_3}(\pm u_{k_{1,2}}\pm u_{k_{2}}u_{k_{1}})$ (the second summand can absent, if $k_2< k_1$), finally, on the place $(-\gamma_3,-\gamma_3)$ there is $d_{-\gamma_3}(1\pm u_{k_{1,2}}t_{k_{1,2}}\pm u_{k_{1}}t_{k_{1}})$. From these five equations with five variables (from radical) we uniquely can define values of variables, after that we suppose that $d_{-\gamma_1}, d_{-\gamma_2}, d_{-\gamma_3}, t_1, t_2, t_{1,2},u_1, u_2,u_{1,2}$ are known.

Suppose now that we know  numbers $t_i, u_j$ for all indices corresponding to the roots of the form $\gamma_p-\gamma_q$, $1\leqslant p,q< s$, all $d_{-\gamma_r}$, $1\leqslant r<s$, and also $s\leqslant l+1$.
Consider the positions $(-\gamma_1,-\gamma_s)$, $(-\gamma_s,-\gamma_1)$, $(-\gamma_2,-\gamma_s)$, $(-\gamma_s,-\gamma_2)$, \dots, $(-\gamma_{s-1},-\gamma_s)$,
$(-\gamma_s,-\gamma_{s-1})$ and $(-\gamma_s,-\gamma_{s})$ in the matrix~$X$. Clear that on every position $(-\gamma_i,-\gamma_s)$, $1\leqslant i<s$, we have the sum of the number $t_p$ (where $p$ is the number of the root $\gamma_i-\gamma_s$, if it is a root) and products of different numbers $t_a,u_b$, where only one number in the product is not known yet, and all others numbers are known, all of them are from the radical, multiplied by $d_{-\gamma_i}$. The same picture  is on the positions $(-\gamma_s,-\gamma_i)$, $1\leqslant i<s$, but there the single element (without multipliers) is not $t_p$, but $u_p$. On the last place there is $d_{-\gamma_s}(1+\Sigma)$, where $\Sigma$ is also a sum of described type. Therefore we have the  number  of  (not uniform) linear equations, greater to~$1$ than the number of roots $\pm (\gamma_i-\gamma_s)$, with the same number of variables, in every equation precisely one variable has an invertible coefficient, other coefficients are from the radical, for different equations such variable are different. Clear that such system of equation has the solution, and it it unique. Consequently, we have made the induction step and now we know the numbers $t_i, u_j$ for all indices corresponding to the roots of the form $\gamma_p-\gamma_q$, $1\leqslant p,q\leqslant s$, and also $d_{-\gamma_s}$.

 After the $l+1$-th step ($s=l+1$) we know all $d_{-\gamma_1},\dots,d_{-\gamma_{l+1}}$, that uniquely define $\lambda, s_1,\dots, s_l$. After that we know all $d_{-\gamma_i}$, $l+1<i\leqslant \gamma_k$.
 On following steps we do not consider the last position of the form $(-\gamma_i,-\gamma_{i})$,  the  number of equations and variables decreases to one.

On the last step we know numbers $t_i, u_j$ for all indices corresponding to the roots of the form $\gamma_p-\gamma_q$, $1\leqslant p,q
\leqslant k$. Consider now in the matrix $X$ the positions $(-\gamma_1,h_{\gamma_1})$, $(h_{\gamma_1},-\gamma_1)$, $(-\gamma_2,h_{\gamma_2})$, $(h_{\gamma_2},-\gamma_2)$, \dots, $(-\gamma_k,h_{\gamma_k})$, $(h_{\gamma_k},-\gamma_k)$. Completely similar to the previous arguments we can find all coefficients  $t$ and $u$, corresponding to the roots $\pm \gamma_1,\dots, \pm \gamma_k$.

After that  we know all coefficients in the product;  moreover, in the proof of the lemma we have described the correspondence between them and positions of the matrix~$X$.

Therefore Lemma is completely proved.
\end{proof}

Now we return to our main proof.
Recall that we work with the matrix~$C$, equivalent to the unit matrix modulo radical, and mapping elementary Chevalley group to itself.

For every root $\alpha\in\Phi$ there is the equation
\begin{equation}\label{osn_eq}
C x_{\alpha}(1)C^{-1}=x_{\alpha}(1)\cdot g_\alpha,\quad g_\alpha\in
G_{\ad}(\Phi,R,J).
\end{equation}
Every element $g_\alpha \in G_{\ad}(\Phi,R,J)$ can be decomposed into the product
\begin{equation}\label{razl_rad}
 t_{\alpha_1}(1+a_1)\dots
t_{\alpha_l}(1+a_l)x_{\alpha_1}(b_1)\dots
x_{\alpha_m}(b_m)x_{\alpha_{-1}}(c_1)\dots x_{\alpha_{-m}}(c_m),
\end{equation}
where $a_1,\dots,a_l,b_1,\dots,b_m,c_1,\dots, c_m\in J$ (see,
for example,~\cite{Abe1}).

Let $C=E+X=E+(x_{i,j})$. Then for every root~$\alpha\in \Phi$
we can write a matrix equation~\ref{osn_eq} with variables $x_{i,j},
a_1,\dots,a_l,b_1,\dots,b_m,c_1,\dots, c_m$, every of them is from the radical.

We change these equations.
We consider the matrix~$C$ and ``imagine'', that it is a matrix from Lemma~\ref{prod2}. Then by some its concrete $n+1$~positions we can  ``reproduce'' all coefficients $\lambda, s_1,\dots,s_l, t_1,\dots,t_m,u_1,\dots, u_m$ in decomposition of matrix in the product from Lemma~\ref{prod2}. In the result we obtain some matrix $D$ from
elementary Chevalley group, every its coefficient is some (known) function of coefficients of the matrix~$C$.
Change now the equations~\eqref{osn_eq} to the equations
\begin{equation}\label{fol_eq}
D^{-1}C x_{\alpha}(1)C^{-1}D=x_{\alpha}(1)\cdot {g_\alpha}',\quad {g_\alpha}'\in
G_{\ad}(\Phi,R,J).
\end{equation}
We again get matrix equations but with variables $y_{i,j},
a_1',\dots,a_l',b_1',\dots,b_m',c_1',\dots, c_m'$, every of them also is from the radical, and every $y_{p,q}$ is some known function from  $x_{i,j}$. We denote the matrix $D^{-1}C$ by~$C'$.

We need to show that a solution exists only if all variables with  primes are zeros. Some $x_{i,j}$ also have to be zeros, and others disappear in equations. Since the equations are too complicated we shall consider the linearized system of equations. It is sufficient to show that all variables that do not disappear in the linearized system (suppose that there are  $q$~such variables) are members of some subsystem consisting from $q$ linear equations with an invertible (in the ring~$R$) determinant.

in other words, from the matrix equations we shall show step by step that all variables are zeros.

Clear that linearization of the product $Y^{-1}(E+X)$ gives some matrix $E+(z_{i,j})$, that have zeros on all positions described in Lemma~\ref{prod2}.

To find the final form of the linearized system we write it as follows:
\begin{multline*}
(E+Z)x_\alpha(1) =x_\alpha(1)(E+a_1T_1+a_1^2\dots)\dots
(E+a_lT_l+a_1^2\dots)\cdot\\
\cdot(E+b_1X_{\alpha_1}+b_1^2X_{\alpha_1}^2/2)\dots
(E+c_mX_{-\alpha_1}+c_m^2X_{-\alpha_m}^2/2)(E+Z),
\end{multline*}
where $X_\alpha$ the corresponding Lie algebra element in the adjoint representation, the matrix $T_i$ is diagonal, has on the diagonal~$p$ at the place corresponding to the vector $v_k$ iff in decomposition of $\alpha_k$ in the sum of simple roots the root
 $\alpha_i$ entries in this decomposition  $p$~times ($p$
can be zero or negative); on the places corresponding to the basis vectors  $V_j$, the matrix has zeros on the diagonal.

Then the linearized system is
$$
Zx_{\alpha}(1)-x_{\alpha}(1)(Z+a_1T_1+\dots+a_lT_l+b_1X_{\alpha_1}+\dots+c_mX_{\alpha_m})=0.
$$
Clear that for simple roots $\alpha_i$ the summand $b_i X_{\alpha_i}$ is absent; besides, $2m$ fixed elements of~$Z$ are zeros.
This equation can be written for every $\alpha\in
\Phi$ (naturally, with other  $a_j, b_j, c_j$), and can be written only for the roots generating the Chevalley groups, i.\,e., for $\alpha_1,\dots,
\alpha_l, -\alpha_1, \dots, -\alpha_l$. The number free variables is not changed:

$$
\begin{cases}
Zx_{\alpha_1}(1)-x_{\alpha_1}(1)(Z+a_{1,1}T_1+\dots +a_{l,1}T_l+\\
\ \ \ \ \
+b_{1,1}X_{\alpha_1}+\dots+b_{m,1}X_{\alpha_m}
+c_{1,1}X_{-\alpha_1}+\dots+c_{m,1}X_{-\alpha_m})=0;\\
\dots\\
Zx_{\alpha_l}(1)-x_{\alpha_l}(1)(Z+a_{1,l}T_1+\dots +a_{l,l}T_l+\\
\ \ \ \ \
+b_{1,l}X_{\alpha_1}+\dots+X_{\alpha_m}b_{m,1}X_{\alpha_m}
+c_{1,l}X_{-\alpha_1}+\dots+c_{m,l}X_{-\alpha_m})=0;\\
Zx_{-\alpha_1}(1)-x_{-\alpha_1}(1)(Z+a_{1,l+1}T_1+\dots+a_{l,l+1}T_l+\\
\ \ \ \ \
+b_{1,l+1}X_{\alpha_1}+\dots+b_{m,l+1}X_{\alpha_m}
+c_{1,l+1}X_{-\alpha_1}+\dots+c_{m,l+1}X_{-\alpha_m})=0;\\
\dots\\
Zx_{-\alpha_l}(1)-x_{-\alpha_l}(1)(Z+a_{1,2l}T_1+\dots+a_{l,2l}T_l+\\
\ \ \ \ \
+b_{1,2l}X_{\alpha_1}+\dots+b_{m,2l}X_{\alpha_m}
+c_{1,2l}X_{-\alpha_1}+\dots+c_{m,2l}X_{-\alpha_m})=0.
\end{cases}
$$

Suppose that we fixed the obtained linear uniform system of equations with described $n+1$ zero positions. Recall that our aim is to show that all values  $z_{i,j}$, $a_{s,t}, b_{s,t}, c_{s,t}$ are zeros.

At first we consider the pair of equations with numbers~$1$ and~$l+1$.
Clear that all other corresponding pairs for the long roots have the same construction (since there is a substitution of basis that moves any long root to any other long root).

We can rename basis so that the matrices $x_{\alpha_1}(1)$ and
$x_{-\alpha_1}(1)$ have the forms of three diagonal blocks: the first one is corresponded to the basis vectors $\{ x_{\alpha_1}, x_{-\alpha_1},
h_1, h_2\}$, the second one corresponded to $\{ x_\alpha\mid \langle \alpha_1,
\alpha\rangle=\pm 1\}$, and the third corresponded to $\{ x_\alpha\mid \langle
\alpha_1,\alpha\rangle=0; h_3,\dots, h_l\}$. Also the third block is just a unit matrix. Let

$$
x_{\alpha_1}(1)=\begin{pmatrix} A& 0& 0\\
0& B& 0\\
0& 0& E
\end{pmatrix},\quad
x_{-\alpha_1}(1)=\begin{pmatrix} A'& 0& 0\\
0& B'& 0\\
0& 0& E
\end{pmatrix}
$$
and suppose that all matrices $Y$ under consideration also have the similar block form:
$$
Y=\begin{pmatrix} Y_{11}& Y_{12}& Y_{13}\\
Y_{21}& Y_{22}& Y_{23}\\
Y_{31}& Y_{32}& Y_{33}
\end{pmatrix}.
$$

Denote the matrix
$a_{1,1}T_1+a_{2,1}T_2+\dots+a_{l,1}T_l+b_{1,1}X_{\alpha_1}+
\dots+b_{m,1}X_{\alpha_m}+c_{1,1}X_{-\alpha_1}+\dots+c_{m,1}X_{-\alpha_m}$
by $D$. Note that from the basis construction  $D_{1,3}=D_{3,1}=0$.
Then the first equation is
$$
\begin{pmatrix}
Z_{11}A& Z_{12}B& Z_{13}\\
Z_{21}A& Z_{22}B& Z_{23}\\
Z_{31}A& Z_{32}B& Z_{33}
\end{pmatrix}=
\begin{pmatrix}
AZ_{11}& AZ_{12}& AZ_{13}\\
BZ_{21}& BZ_{22}& BZ_{23}\\
Z_{31}& Z_{32}& Z_{33}
\end{pmatrix}+
\begin{pmatrix}
AD_{11}& AD_{12}& 0\\
BD_{21}& BD_{22}& BD_{23}\\
0 & D_{32}& D_{33}
\end{pmatrix}.
$$
From the written equality directly follows that $D_{33}=0$.
Therefore, $a_{3,1}=\dots=a_{l,1}=0$, $b_{k,1}=c_{k,1}=0$ при
$\langle \alpha_1, \alpha_k\rangle=0$.

Now consider a positive  root $\alpha$ such that $\beta=\alpha+\alpha_1\in \Phi$ and $\alpha,\beta$ are from the first $l+1$ members of the sequence $\gamma_1,\dots, \gamma_k$, namely, $\gamma_l=e_2+e_l$  (similar pairs of roots in the sequence $\gamma_1,\dots, \gamma_k$ can be found also for $\alpha_2,\dots, \alpha_l$ according to the property~3 of the sequence $\gamma_1,\dots, \gamma_k$).

Consider the basis part $\alpha,-\alpha,\beta,-\beta$. For the matrices $x_{\alpha_1}(1)$ and
$x_{-\alpha_1}(1)$ this basis part is a direct summand, therefore one can consider it independently. By the construction we know that $z_{-\alpha,-\alpha}=z_{-\beta,-\beta}=z_{-\alpha,-\beta}=z_{\beta,-\alpha}=0$. We have the condition
\begin{multline*}
\begin{pmatrix}
z_{\alpha,\alpha}& z_{\alpha,-\alpha}& z_{\alpha,\beta}& z_{\alpha,-\beta}\\
z_{-\alpha,\alpha}& 0& z_{-\alpha,\beta}& 0\\
z_{\beta,\alpha}& z_{\beta,-\alpha}& z_{\beta,\beta}& z_{\beta,-\beta}\\
z_{-\beta,\alpha}& 0& z_{-\beta,\beta}& 0
\end{pmatrix}
\begin{pmatrix}
1& 0& 0& 0\\
0& 1& 0& -1\\
1& 0& 1& 0\\
0& 0& 0& 1
\end{pmatrix}=\\ =
\begin{pmatrix}
1& 0& 0& 0\\
0& 1& 0& -1\\
1& 0& 1& 0\\
0& 0& 0& 1
\end{pmatrix}
\left(\begin{pmatrix}
z_{\alpha,\alpha}& z_{\alpha,-\alpha}& z_{\alpha,\beta}& z_{\alpha,-\beta}\\
z_{-\alpha,\alpha}& 0& z_{-\alpha,\beta}& 0\\
z_{\beta,\alpha}& z_{\beta,-\alpha}& z_{\beta,\beta}& z_{\beta,-\beta}\\
z_{-\beta,\alpha}& 0& z_{-\beta,\beta}& 0
\end{pmatrix}+
\begin{pmatrix}
a_{\alpha} & 0& -c_{1,1}& 0\\
0& -a_{\alpha}& 0& -b_{1,1}\\
b_{1,1}& 0& a_{\beta}& 0& 0\\
0& c_{1,1}& 0& -a_{\beta}
\end{pmatrix}\right) .
\end{multline*}

As the result we obtain that the following matrix is equal to zero:
$$
\begin{pmatrix}
z_{\alpha,\beta}-a_{\alpha}& 0& c_{1,1}& z_{\alpha,-\alpha}\\
z_{-\alpha,\beta}+z_{-\beta,\alpha}& c_{1,1}+a_{\alpha}& z_{-\beta,\beta}& b_{1,1}-a_{\beta}\\
z_{\beta,\beta}-z_{\alpha,\alpha}-a_{\alpha}-b_{1,1}& -z_{\alpha,-\alpha}& -z_{\alpha,\beta}+c_{1,1}-a_{\beta}&
-z_{\beta,-\alpha}-z_{\alpha,-\beta}\\
z_{-\beta,\beta}& -c_{1,1}& 0& a_{\beta}
\end{pmatrix}.
$$
So we see that $a_{\alpha}=a_{\beta}=b_{1,1}=c_{1,1}=0$. Since we know that $a_{3,1}=\dots=a_{l,1}=0$,
we have $a_{1,1}=a_{2,1}=0$.

Consider such a root $\alpha\in \Phi^+$, that $\alpha+\alpha_1\in \Phi^+$. In the sequence $\gamma_1,\dots, \gamma_k$ we find such roots $\gamma_p$ and $\gamma_q$, that $\gamma_q-\gamma_p=\alpha$. The roots $\alpha$ can be of three following forms: $\alpha=e_2-e_i$, $3\leqslant i\leqslant l$; $\alpha=e_2+e_i$,  $3\leqslant i\leqslant l$; $\alpha=e_2$. For these three cases only two situations can be possible:

1) the root $\gamma_p$ is orthogonal to the root~$\alpha_1$, and the root $\gamma_q$ is not orthogonal. Then $\gamma'=\gamma_q+\alpha_1$ is also a root, we can consider the basis part $\gamma_p,-\gamma_p,\gamma_q,-\gamma_q, \gamma',-\gamma'$, that is an invariant direct summand for $x_{\alpha_1}(1)$. On this basis part
$$
x_{\alpha_1}(1)=\begin{pmatrix}
1& 0& 0& 0& 0& 0\\
0& 1& 0& 0& 0& 0\\
0& 0& 1& 0& 0& 0\\
0& 0& 0& 1& 0& -1\\
0& 0& 1& 0& 1& 0\\
0& 0& 0& 0& 0& 1
\end{pmatrix},
$$
the matrix $Z$ has the positions   $z_{-\gamma_p,-\gamma_q}$ and $z_{-\gamma_q,-\gamma_p}$ equal to zero (by Lemma~\ref{prod2}), the matrix $D$ is
$$
\begin{pmatrix}
0& 0& -c_{\alpha,1}& 0& -c_{\alpha+\alpha_1,1}& 0\\
0& 0& 0& -b_{\alpha,1}& 0& -b_{\alpha+\alpha_1}\\
b_{\alpha,1}& 0& 0& 0& 0& 0\\
0& c_{\alpha,1}& 0& 0& 0& 0\\
b_{\alpha+\alpha_1,1}& 0& 0& 0& 0& 0\\
0& c_{\alpha+\alpha_1,1}& 0& 0& 0& 0
\end{pmatrix}.
$$

Now from the main condition the following matrix is zero:
{\tiny
$$
\begin{pmatrix}
0& 0& z_{\gamma_p,\gamma'}+c_{\alpha,1}& 0& c_{\alpha+\alpha_1,1}& -z_{\gamma_p,-\gamma_q}\\
0& 0& z_{-\gamma_p,\gamma'}& b_{\alpha,1}& 0& b_{\alpha+\alpha_1,1}\\
-b_{\alpha,1}& 0& z_{\gamma_q,\gamma'}& 0& 0& -z_{\gamma_q,-\gamma_q}\\
z_{-\gamma',\gamma_p}& z_{-\gamma',-\gamma_p}-c_{\alpha,1}+c_{\alpha+\alpha_1,1}& z_{-\gamma_q,\gamma'}+z_{-\gamma',\gamma_q}& z_{-\gamma',-\gamma_q}& z_{-\gamma',\gamma'}&
z_{-\gamma',-\gamma'}-z_{-\gamma_q,-\gamma_q}\\
-z_{\gamma_q,\gamma_p}-b_{\alpha,1}-b_{\alpha+\alpha_1,1}& -z_{\gamma_q,-\gamma_p}& z_{\gamma',\gamma'}-z_{\gamma_q,\gamma_q}& -z_{\gamma_q,-\gamma_q}& -z_{\gamma_q,\gamma'}&
-z_{\gamma',-\gamma_q}-z_{\gamma_q,-\gamma'}\\
0& -c_{\alpha+\alpha_1}& z_{-\gamma',\gamma'}& 0& 0& z_{-\gamma',-\gamma_q}
\end{pmatrix}.
$$
}

Easy to see, that $b_{\alpha,1}=b_{\alpha+\alpha_1,1}=c_{\alpha+\alpha_1,1}=0$.

2) In the second case we have $\gamma''=\gamma_p-\alpha_1\in \Phi^+$, $\langle \gamma_q,\alpha_1\rangle =0$. This case is similar to the previous one, and on the basis part $\gamma_p,-\gamma_p,\gamma_q,-\gamma_q,\gamma'',-\gamma''$
the following matrix is zero:
{\tiny
$$
\begin{pmatrix}
-z_{\gamma'',\gamma_p}& -z_{\gamma_p,-\gamma''}-z_{\gamma'',-\gamma_p}& -z_{\gamma'',\gamma_q}+c_{\alpha,1}+c_{\alpha+\alpha_1,1}& -z_{\gamma'',-\gamma_q}& z_{\gamma_p,\gamma_p}-z_{\gamma'',\gamma''}& -z_{\gamma'',-\gamma''}\\
 0& -z_{-\gamma_p,-\gamma''}& 0& b_{\alpha,1}& z_{-\gamma_p,\gamma_p}& 0\\
 -b_{\alpha,1}& -z_{\gamma_q,-\gamma''}& 0& 0& z_{\gamma_q,\gamma_p}-b_{\alpha+\alpha_1,1}& 0\\
 0& -z_{-\gamma+q,-\gamma''}-c_{\alpha,1}& 0& 0& z_{-\gamma_q,\gamma_p}& -c_{\alpha+\alpha_1,1}\\
 0& -z_{\gamma'',-\gamma''}&  c_{\alpha+\alpha_1,1}& 0& z_{\gamma'',\gamma_p}& 0\\
 z_{-\gamma_p,\gamma_p}& -z_{-\gamma'',-\gamma''}+z_{-\gamma_p.-\gamma_p}& z_{-\gamma_p,\gamma_q}&
  b_{\alpha+\alpha_1,1}-b_{\alpha,1}& z_{-\gamma'',\gamma_p}+z_{-\gamma_p,\gamma''}& z_{-\gamma_p,-\gamma''}
  \end{pmatrix}.
  $$
}

We again have $b_{\alpha,1}=b_{\alpha+\alpha_1,1}=c_{\alpha+\alpha_1,1}=0$.

Considering a pair of roots $\gamma_s,\gamma_t$ with the property $\gamma_t-\gamma_s=\alpha+\alpha_1$ we similarly obtain two cases and $c_{\alpha,1}=0$.

Therefore we proved that $a_{i,1}=0$ for every $i=1,\dots,l$, $b_{\alpha,1}=c_{\alpha,1}=0$ for every $\alpha\in \Phi$. Consequently, the matrix~$Z$ commutes with $x_{\alpha_1}(1)$.

Similarly from other $l-2$ conditions we obtain that the matrix $Z$ commutes with all $x_{\pm\alpha_i}(1)$, $i=1,\dots,l-1$.

Clear that it means that $Z$ commutes with all $X_{\pm \alpha_i}$, $i=1,\dots, l-1$.

Now we need to consider only to conditions: for $x_{\pm e_l}(1)$.

Similarly to the long roots the matrices $x_{\pm e_l}(1)$ are decomposed into diagonal blocks. As we know from the previous sections, these blocks are corresponded to the basis parts $\{ v_{\pm e_i\pm e_j}\mid 1\leqslant i< j< l\}\cup \{V_{h_1},\dots, V_{h_{l-2}}\} $ (both matrices are identical on it), $\{ v_{\pm e_i \pm e_l}, v_{\pm e_i}\}$, $1\leqslant i< l$, $\{ v_{e_l}, v_{-e_l}, V_{h_{l-1}}, V_{h_l}\}$.

Similarly to the situation with long roots we can get $D_{3,3}=0$, i.\,e., $a_{1,l}=\dots=a_{l-2,l}=0$, $b_{k,l}=c_{k,l}=0$ for
$\alpha_k=\pm e_i\pm e_j$, $1\leqslant i<j<l$.

Now consider the roots $\gamma_{l+1}=e_2$ and $\gamma_l=\gamma_{l+1}+\alpha_l=e_2+e_l$.
Consider the basis part $e_2-e_l,e_l-e_2,e_2+e_l, -e_2-e_l, e_2, -e_2$. For the matrices $x_{\alpha_l}(1)$ and
$x_{-\alpha_l}(1)$ this basis  is a direct summand, therefore it can be considered independently. By the construction we know that $z_{-e_2,-e_2}=z_{-e_2-e_l,-e_2-e_l}=z_{-e_2,-e_2-e_l}=z_{-e_2-e_l,-e_2}=0$.

Besides, if $Z$ commutes with all $X_{\pm \alpha_1},\dots, X_{\pm \alpha_{l-1}}$, then $Z$ commutes with $X_{e_2-e_l}$, $X_{e_l-e_2}$. Considering the basis part $\{ v_{e_2-e_l}, v_{e_l-e_2}, v_{e_2+e_l}, v_{-e_2-e_l}, V_{h_1},\dots, V_{h_l}\}$, we see that $z_{\alpha,\beta}=0$ for all pairs $\langle \alpha,\beta\rangle$, where either $\alpha=\pm (e_2-e_l)$, $\beta=\pm (e_2+e_l)$, or $\alpha=\pm (e_2+e_l)$, $\beta=\pm (e_2-e_l)$.

Now we have the condition
{\tiny
\begin{multline*}
\begin{pmatrix}
z_{e_2-e_l,e_2-e_l}& z_{e_2-e_l,e_l-e_2}& 0& 0&  z_{e_2-e_l,e_2}& z_{e_2-e_l,-e_2}\\
z_{e_l-e_2,e_2-e_l}& z_{e_l-e_2,e_l-e_2}& 0& 0&  z_{e_l-e_2,e_2}& z_{e_l-e_2,-e_2}\\
0& 0& z_{e_2+e_l,e_2+e_l}& z_{e_2+e_l,-e_2-e_l}&  z_{e_2+e_l,e_2}& z_{e_2+e_l,-e_2}\\
0& 0& z_{-e_2-e_l,e_2+e_l}& 0&  z_{-e_2-e_l,e_2}& 0\\
z_{e_2,e_2-e_l}& z_{e_2,e_l-e_2}& z_{e_2,e_2+e_l}& z_{e_2,-e_2-e_l}&  z_{e_2,e_2}& z_{e_2,-e_2}\\
z_{-e_2,e_2-e_l}& z_{-e_2,e_l-e_2}& z_{-e_2,e_2+e_l}& 0&  z_{-e_2,e_2}& 0
\end{pmatrix}
\begin{pmatrix}
1& 0& 0& 1& 0& 1\\
0& 1& 0& 0& 0& 0\\
0& 0& 1& 1& 0& -1\\
0& 0& 0& 1& 0& 0\\
0& -2& 0& 0& 1& 0\\
0& 0& 0& 2& 0& 1
\end{pmatrix}=\\ =
\begin{pmatrix}
1& 0& 0& 1& 0& 1\\
0& 1& 0& 0& 0& 0\\
0& 0& 1& 1& 0& -1\\
0& 0& 0& 1& 0& 0\\
0& -2& 0& 0& 1& 0\\
0& 0& 0& 2& 0& 1
\end{pmatrix}
\left(\begin{pmatrix}
z_{e_2-e_l,e_2-e_l}& z_{e_2-e_l,e_l-e_2}& 0& 0&  z_{e_2-e_l,e_2}& z_{e_2-e_l,-e_2}\\
z_{e_l-e_2,e_2-e_l}& z_{e_l-e_2,e_l-e_2}& 0& 0&  z_{e_l-e_2,e_2}& z_{e_l-e_2,-e_2}\\
0& 0& z_{e_2+e_l,e_2+e_l}& z_{e_2+e_l,-e_2-e_l}&  z_{e_2+e_l,e_2}& z_{e_2+e_l,-e_2}\\
0& 0& z_{-e_2-e_l,e_2+e_l}& 0&  z_{-e_2-e_l,e_2}& 0\\
z_{e_2,e_2-e_l}& z_{e_2,e_l-e_2}& z_{e_2,e_2+e_l}& z_{e_2,-e_2-e_l}&  z_{e_2,e_2}& z_{e_2,-e_2}\\
z_{-e_2,e_2-e_l}& z_{-e_2,e_l-e_2}& z_{-e_2,e_2+e_l}& 0&  z_{-e_2,e_2}& 0
\end{pmatrix}+\right.\\
+\left.\begin{pmatrix}
a_{l-1,l} & 0& 0& 0& 0& b_{l,l}\\
0& -a_{l-1,l}& 0& 0& 0& c_{l,l}&0\\
0& 0& 2a_{l,l}+a_{l-1,l}& 0& -b_{l,l}& 0& 0\\
0& 0& 0& -2a_{l,l}-a_{l-1,l}& 0&  0& -c_{l,l}\\
0& -2b_{l,l}& 2c_{l,l}& 0& a_{l,l}+a_{l-1,l}& 0\\
-2c_{l,l}&0& 0& 2b_{l,l}& 0& -a_{l,l}-a_{l-1,l}
\end{pmatrix}\right) .
\end{multline*}
}
It directly implies $a_{l-1,l}=a_{l,l}=b_{l,l}=c_{l,l}=0$.

 For the roots $e_i-e_l$, $e_i+e_l$, $e_i$, $1\leqslant i< l$,  we act in the same manner as above, when we considered the first condition, and we obtain that all coefficients $b_{\alpha,l}$ and $c_{\alpha,l}$ are zeros.

Similarly we consider the condition with $x_{-\alpha_l}(1)$.

Therefore $Z$ commutes with all $x_{\alpha}(1)$, $\alpha\in \Phi$,
Since $Z$ has zeros on the diagonal (by the construction) it is zero.
 Theorem is proved.
\end{proof}

\section{Proof of the main theorem (Theorem 1).}

\begin{lemma}\label{tor}
If $\overline t\in G_{\ad}(\Phi,R)$ is a torus element $(\Phi$ is one of the systems under consideration, $R$ is a local commutative ring\emph{)}, then there exists such a torus element $t\in G_\pi(\Phi,S)$ that a ring $S$ contains~$R$, $t$ lies in the normalizer of $G_\pi(\Phi,R)$, under factorization of $G_\pi(\Phi,S)$ by its center $t$ gives~$\overline t$.
\end{lemma}

\begin{proof}
Clear that it is sufficient to prove the lemma statement for basic elements of the torus $T_{\ad}(\Phi,R)$. Since all roots of the same length are conjugate, we can consider only two torus elements: $\overline t_1=\chi_{\alpha_1}(r_1)$ and $\overline t_l=\chi_{\alpha_l}(r_l)$.
Torus elements under consideration acts on the elementary subgroup as follows:
$\overline t_i x_{\alpha}(s) \overline t_i^{-1}=x_{\alpha}(r_i^{k_i}s)$, where $\alpha=k_i\alpha_i+\beta$, $\alpha_i$ does not entry in the decomposition $\beta$ in the sum of simple roots.

Clear that we need only to construct such an extension  $S$ of~$R$  and torus elements of the group $G_\pi(\Phi,S)$, $t_1$ and $t_l$, that $t_i x_{\alpha}(s) t_i^{-1}= \overline t_i x_\alpha(s) \overline t_i^{-1}$ for all $\alpha\in \Phi$, $s\in R$.

Clear that if the constructed torus elements act correctly on simple roots, it automatically acts correctly on all roots.

Consider a ring~$S$, that obtained from~$R$ by adding a root of the second power from~$r_1$ and $r_l$. Denote these roots by~$s_1$ and $s_l$ respectively. Then the obtained elements are
$$
t_1=h_{\alpha_1}(s_1^2)h_{\alpha_2}(s_1^2)\dots h_{\alpha_{l-1}}(s_1^2)h_{\alpha_l}(s_1)
$$
and
$$
t_l=h_{\alpha_1}(s_l^2)h_{\alpha_2}(s_l^4)\dots h_{\alpha_{l-1}}(s_l^{2(l-1)})h_{\alpha_l}(s_l^l),
$$
it is checked by direct calculations.

Therefore, all obtained extensions of~$R$ are found, elements $ t$ are constructed.
\end{proof}

Let us prove now the main theorem (Theorem~\ref{main}).

\begin{proof}
The case when the Chevalley group is elementary adjoint,
evidently follows from Theorems \ref{old} and~\ref{norm}.

Suppose now that we have any other elementary Chevalley group
$E_\pi(\Phi,R)$ and some its automorphism~$\varphi$. Its quotient by the center is isomorphic to the elementary adjoint group $E_{\ad}(\Phi,R)$,
so we have an automorphism $\overline \varphi$ of the group
$E_{\ad}(\Phi,R)$. Such an automorphism is decomposed to
$$
\overline \varphi = \overline \rho \circ
\varphi_{\overline g},
$$
where $\overline \rho$ is a ring automorphism, $\varphi_{\overline g}$ is a conjugation by $\overline g\in G_{\ad}(\Phi,R)$. Note that an automorphism $\overline \rho$  can be easily changed
to a ring automorphism $\rho$  of the group
$E_\pi(\Phi,R)$ such that $\rho$  on equivalence classes of the group $E_{\pi}(\Phi,R)$ by its center acts in the same way as $\overline \rho$.

Now consider an automorphism $\varphi_{\overline g}$. Note that
$\overline g= \overline t \overline e$, where  $\overline t\in
T_{\ad}(\Phi, R)$, $\overline e\in E_{\ad}(\Phi, R)$. For an element
$\overline e$ we can find such $e\in E_{\pi}(\Phi,R)$ that the image
$e$ under factorization by its center is equal to~$\overline e$. An element
 $\overline t$ has its inverse image $t\in T_{\pi}(\Phi, S)$, where $S$ is a ring obtained from~$R$ by adding some scalars (see Lemma~\ref{tor}). Also $t$ normalizes $E_\pi(\Phi,R)$.
Consider now $g=te\in G_{\pi}(\Phi,S)$. Clear that under factorization the group $E_\pi (\Phi, R)$ by its center an automorphism
$\varphi_g$ gives us an automorphism~$\varphi_{\overline g}$.

Now consider an automorphism
$$
\psi=\varphi_{g^{-1}}\circ \rho^{-1}\circ \varphi.
$$
It is an automorphism of $E_\pi(\Phi,R)$, under factorization by the center it gives an identical automorphism of the quotient group. Therefore $\psi$ is a central automorphism.

But since an elementary Chevalley group in the case under consideration is its proper commutant, we have that every central automorphism is identical.

Consequently we proved the theorem for all elementary Chevalley groups of types under consideration.
Namely, we proved that every automorphism of $E_{\pi}(\Phi,R)$ is a composition of ring and inner (but not strictly inner) automorphisms.

Now suppose that we have a Chevalley group $G_\pi(\Phi,R)$ and its automorphism~$\varphi$.
Since an elementary group
$E_\pi(\Phi,R)$ is characteristic (commutant) in
$G_\pi(\Phi,R)$, then $\varphi$ is simultaneously an automorphism of the elementary subgroup. On the elementary subgroup it is the composition $\rho\circ \varphi_g$, $g\in
G_{\pi}(\Phi,S)$, and $g=te$, where $e\in E_\pi(\Phi,R)$, $t\in T_\pi (\Phi,S)$. The first two automorphisms are clearly extended to the automorphisms of the whole group $G_{\pi}(\Phi,R)$, and the second one is an automorphism of this group, since the torus elements commutes. Then the composition $\psi=\varphi_{g^{-1}}\circ
\rho^{-1}\circ \varphi$ is an automorphism of
$G_\pi(\Phi,R)$, that acts identically on the elementary subgroup.
Since $G_\pi(\Phi,R)=T_\pi(\Phi,R)\cdot E_{\pi}(\Phi,R)$, we need only to understand how  $\psi$ acts on torus elements. The element
$t^{-1}\varphi(t)\in G_\pi(\Phi,R)$ lies in the center of
$E_\pi(\Phi,)$, therefore, in the center of $G_\pi(\Phi,R)$.
Consequently the automorphism $\psi$ is central.

Therefore, Theorem~\ref{main} is proved for all Chevalley groups under consideration.
\end{proof}

\end{document}